\documentclass[reqno,a4paper,12pt]{amsart}

\usepackage[all]{xy}
\usepackage{amsfonts}
\usepackage[mathcal]{eucal}
\usepackage{eufrak}
\usepackage{amssymb}
\usepackage{amsmath}
\usepackage{color}
\usepackage[pagebackref,colorlinks]{hyperref}

\newcommand{\ra}{\rightarrow}

\newcommand{\nono}{\nonumber}

\newcommand{\mi}{\setminus}
\newcommand{\lra}{\longrightarrow}
\newcommand{\ov}{\overline}

\newcommand{\mas}{\mathcal{S}}
\newcommand{\maj}{\mathcal{G}}
\newcommand{\mfn}{\mathfrak{N}}
\newcommand{\dlambda}{\delta}
\newcommand{\id}{\operatorname{id}}
\newcommand{\lbr}{[}
\newcommand{\rbr}{]}
\newcommand{\nn}{\mathbb N}
\newcommand{\zz}{\mathbb Z}

\newcommand{\conggr}{\cong_{\mathrm{gr}}}

\newcommand{\Ga}{\Gamma}
\newcommand{\ga}{\gamma}
\newcommand{\al}{\alpha}

\newcommand{\de}{\delta}
\newcommand{\la}{\lambda}

\newcommand{\si}{\sigma}
\newcommand{\ep}{\varepsilon}
\newcommand{\De}{\Delta}
\newcommand{\ph}{\varphi}

\newcommand{\sbeq}{\subseteq}
\newcommand{\speq}{\supseteq}

\newcommand{\ti}{\widetilde}

\newcommand\K[1][1]{{\operatorname{K}_{#1}}}
\newcommand\CK[1][1]{\operatorname{CK}_{#1}}
\newcommand\SH{\operatorname{SH^0}}

\newcommand\SK[1][1]{\operatorname{SK}_{#1}}
\newcommand{\Nrd}[1][{}]{{\operatorname{Nrd}_{#1}}}
\newcommand{\Trd}[1][{}]{{\operatorname{Trd}_{#1}}}
\newcommand{\Tr}[1][{}]{{\operatorname{Tr}_{#1}}}

\newcommand{\GL}{\operatorname{GL}}

\newcommand{\Aut}{\operatorname{Aut}}
\newcommand{\chr}{\operatorname{char}}
\newcommand{\ind}{\operatorname{ind}}
\newcommand{\intt}{\operatorname{int}}
\newcommand{\Div}{\operatorname{Div}}
\newcommand{\End}{\operatorname{End}}
\newcommand{\Gal}{\mathcal G al}
\newcommand{\Inn}{\operatorname{Inn}}

\newcommand{\ann}{\operatorname{ann}}
\newcommand{\ddet}{\operatorname{ddet}}

\newcommand{\gr}{\text{gr}}
\newcommand{\im}{\text{im}}

\theoremstyle{plain}
\newtheorem {lemma}{Lemma}[section]
\newtheorem {theorem}[lemma]{Theorem}
\newtheorem {thm}[lemma]{Theorem}
\newtheorem {corollary}[lemma]{Corollary}

\newtheorem {cor}[lemma]{Corollary}
\newtheorem {proposition}[lemma]{Proposition}
\newtheorem {prop}[lemma]{Proposition}

\theoremstyle{remark}
\newtheorem* {remark}{Remark}  
\newtheorem {rem}[lemma]{Remark} 
\newtheorem {remarks}[lemma]{Remarks} 
\newtheorem {example}[lemma]{Example}

\theoremstyle{definition}

\newtheorem {deff}[lemma]{Definition} 

\numberwithin{equation}{section}

\title{$\SK$ of graded division algebras}

\date{}

\begin{thanks}
{The first author acknowledges the support of EPSRC first grant
scheme EP/D03695X/1. The second author would like to thank 
the first author and Queen's University, Belfast for their 
hospitality while the research for this paper was carried out.}
\end{thanks}

\author{R. Hazrat}\address{
Department of Pure Mathematics\\
Queen's University\\
Belfast BT7 1NN\\
United Kingdom} \email{r.hazrat@qub.ac.uk}

\author{A. R. Wadsworth}
\address{
Department of Mathematics\\
University of California at San Diego\\
La Jolla, California 92093-0112\\
U.S.A.}
 \email{arwadsworth@ucsd.edu}

\begin{document}

\begin{abstract}
The reduced Whitehead group $\SK$ of a graded division 
algebra graded by a torsion-free abelian group is studied. 
It is observed that the computations here are much more 
straightforward than in the non-graded setting. Bridges to 
the ungraded case are
then established by the following two theorems: It is
 proved that $\SK$ of a tame valued division algebra
over a henselian field 
coincides with $\SK$ of its associated graded division 
algebra. Furthermore, it is shown that $\SK$ of a graded 
division algebra is isomorphic to $\SK$ of its quotient 
division algebra. The first theorem gives the established 
formulas for the reduced Whitehead group of certain valued 
division algebras in a unified manner, whereas the latter 
theorem covers 
the stability of reduced Whitehead groups, and also describes 
$\SK$ for generic abelian crossed products. 
\end{abstract}

\maketitle
\tableofcontents

\section{Introduction}
Let $D$ be a division algebra with a valuation. To this 
one associates a graded division algebra 
$\gr(D)= \bigoplus_{\ga \in \Ga_D}\gr(D)_\gamma$, where 
$\Ga_D$ is the value group of $D$ and 
the summands $\gr(D)_\gamma$ arise from the filtration on $D$ 
induced by the valuation
(see \S\ref{prel} for details).   As is illustrated in 
\cite{hwcor}, even though computations in the graded setting are 
often easier than working directly with $D$, 
it seems that not  much is lost in passage from $D$ to its 
corresponding graded division algebra $\gr(D)$.  
This has provided motivation to systematically study this 
correspondence, notably by Boulagouaz \cite{boulag}, Hwang, 
Tignol and Wadsworth \cite {hwalg,hwcor,tigwad}, and to
compare certain functors defined on these objects, 
notably the Brauer group.

In particular, the associated graded ring $\gr(D)$ is an Azumaya 
algebra (\cite{hwcor}, Cor.~1.2);  so the reduced norm map 
exists for it, and  one defines the reduced Whitehead group $\SK$ 
for $\gr(D)$ as 
the kernel of the reduced norm map and $\SH$ as its cokernel 
 (see \S\ref{reducedsec}). 
In this paper
 we study these groups for a graded division algebra.

Apart from the work of Panin and Suslin \cite{pansus} on 
$\SH$ for Azumaya algebras over semilocal regular rings 
and \cite{azumayask1} which studies $\SK$ for Azumaya 
algebras over henselian rings, it seems that not much is known 
about these groups in the setting of Azumaya algebras.   
Specializing to division algebras, however, there is an 
extensive literature on the group $\SK$. Platonov \cite{platonov} 
showed that $\SK$ could be non-trivial for certain  
division algebras over  henselian valued fields. He thereby 
provided a series of counter-examples to questions raised in the 
setting of algebraic 
groups, notably the Kneser-Tits conjecture.
  (For surveys on this work and the group $\SK$, see \cite{platsurvey}, \cite{gille},  \cite{merk} or  
\cite{wadval}, \S6.)

In this paper we first study the reduced Whitehead group $\SK$ of a 
graded division algebra whose grade group is   totally ordered 
abelian  (see \S\ref{reducedsec}). It can be  observed that the 
computations here are significantly easier and more transparent
than in the 
non-graded setting.  For a division algebra $D$ finite-dimensional
over a 
henselian valued field $F$, the valuation on $F$ extends uniquely 
to $D$ 
(see Th.~2.1 in \cite{wadval}, or \cite{wad87}), and 
the filtration on $D$ induced by the valuation yields an 
 associated graded division algebra $\gr(D)$. 
Previous work on the subject has shown that this 
transition to graded setting is most ``faithful'' when the 
valuation is tame.  Indeed, in Section~\ref{valued}, we show 
that for a tame valued division algebra $D$ over a henselian 
field, $\SK(D)$ coincides with $\SK(\gr(D))$ (Th.~\ref{sk1prop}).  
Having established this bridge between the 
graded setting and non-graded case, we will easily deduce 
known formulas in the literature for the reduced Whitehead 
group of  certain valued division algebras,  by passing to 
the graded setting;  this shows the utility  of the graded approach 
(see Cor. ~\ref{sk1appl}).

In the other direction, if $E= \bigoplus_{\ga \in \Ga_E}E_\gamma$ is a 
graded division algebra whose grade group $\Ga_E$ is  torsion-free 
abelian, then $E$  
has a quotient division algebra $q(E)$ which has the same 
index as~$E$. The same question on comparing the reduced 
Whitehead groups of these objects can also be raised here.    
It is known that when the 
grade group is $\mathbb  Z$, then $E$
has the simple form of a skew Laurent polynomial ring
$D[x,x^{-1},\varphi]$, where $D$ is a division algebra and 
$\varphi$ is an automorphism of $D$. In this setting the 
quotient division algebra of $D[x,x^{-1},\varphi]$
is $D(x,\varphi)$. In \cite{py}, 
Platonov and {Yanchevski\u\i}  compared $\SK(D(x,\varphi))$ with 
$\SK(D)$. In particular, they showed that if $\varphi$ is an 
inner automorphism then $\SK(D(x,\varphi))\cong \SK(D)$. In 
fact, if $\varphi$ is inner, then $D[x,x^{-1},\varphi]$ is an 
unramified graded  division algebra and we prove that 
$\SK(D[x,x^{-1},\varphi]) \cong \SK(D)$ (Th.~\ref{skunramthm}). 
By combining these, one concludes that the reduced Whitehead 
group of the graded division algebra $D[x,x^{-1},\varphi]$, 
where $\varphi$ is inner, coincides with $\SK$ of its quotient 
division algebra. In Section~\ref{skqd}, we show that this 
is a very special case of stability of $\SK$ for graded 
division algebras; namely, for any graded division algebra 
with torsion-free grade group, the reduced Whitehead 
group coincides with the reduced Whitehead group of its 
quotient division algebra.  This allows us to give a formula for 
$\SK$ for generic abelian crossed product algebras.

The paper is organized as follows: In Section~\ref{prel}, we 
gather  relevant background  on the theory of 
graded division algebras indexed by a totally ordered abelian 
group and establish several homomorphisms needed in the paper.   
Section~\ref{reducedsec}  studies the reduced Whitehead group $\SK$ 
of a 
graded division algebra. We establish 
analogues to Ershov's linked exact 
sequences~\cite{ershov} in the 
graded setting, easily  
deducing formulas for $\SK$ of unramified, totally
ramified, and  semiramified 
graded division algebras. In Section~\ref{valued}, we prove that 
$\SK$ of a tame division algebra over a henselian field 
coincides with $\SK$ of its associated graded division 
algebra. Section~\ref{skqd} is devoted to proving that  $\SK$ of a 
graded division algebra is isomorphic to $\SK$ of is quotient 
division algebra. We conclude the paper with two appendices. 
Appendix~\ref{weddapp} establishes the Wedderburn factorization 
theorem on the setting of 
graded division rings, namely  that the minimal polynomial of a 
homogenous element of a 
graded division ring $E$ splits completely over $E$
(Th.~\ref{domainW}). 
Appendix~\ref{apencongru} provides a complete proof of the 
Congruence Theorem for all tame 
 division algebras over henselian valued fields.
This theorem was originally proved by Platonov for the case of 
complete discrete valuations of rank~1, and it was a key tool 
in his calculations of  $\SK$ for certain valued division algebras.   

\section{Graded division algebras}\label{prel}

In this section we establish notation and 
recall some fundamental facts about graded division 
algebras indexed by a totally ordered abelian group, and 
about their connections with valued division algebras.
In addition, we
 establish some important  homomorphisms relating the group 
structure of a valued division algebra to the group 
structure of its associated graded division algebra.

Let
$R = \bigoplus_{ \ga \in \Gamma} R_{\ga}$ be a
graded ring, i.e., 
  $\Gamma$ is an abelian group,  and $R$ is a 
unital ring such that each $R_{\ga}$ is a
subgroup of $(R, +)$ and 
$R_{\ga} \cdot R_{\de} \subseteq R_{\ga +\de}$ 
for all $\ga, \de \in \Ga$. Set
\begin{itemize}
\item[] $\Gamma_R  \ = \  \{\ga \in \Ga \mid R_{\ga} \neq 0 \}$, 
 \  the grade set of $R$;
\vskip0.05truein
\item[] $R^h  \ = \ \bigcup_{\ga \in
\Ga_{R}} R_{\ga}$,  \ the set of homogeneous elements of $R$.
\end{itemize}
For a homogeneous element of $R$  of degree $\gamma$, i.e., an 
$r \in R_\gamma\mi{0}$, we write $\deg(r) = \gamma$.  
Recall  that $R_{0}$ is a subring of $R$ and that for each $\ga \in
\Ga_{R}$, the group $R_{\ga}$ is a left and right $R_{0}$-module.
A subring $S$ of
$R$ is a \emph{graded subring} if $S= \bigoplus_{ \ga \in
\Gamma_{R}} (S \cap R_{\ga})$.  For example, the 
center of $R$, denoted $Z(R)$, is a graded subring of 
$R$.
If $T = \bigoplus_{ \ga \in
\Gamma} T_\gamma$ is another graded ring, 
a {\it graded ring homomorphism} is a ring homomorphism 
$f\colon R\to T$ with $f(R_\gamma) \subseteq T_\gamma$
for all $\gamma \in \Gamma$.  If $f$ is also bijective, 
it is called a graded ring isomorphism;  we then write
$R\conggr T$.

For a graded ring $R$, a graded left $R$-module $M$ is
a left  $R$-module with a grading ${M=\bigoplus_{\ga \in \Ga'} 
M_{\ga}}$,
where the $M_{\ga}$ are all abelian groups and $\Ga'$ is a 
abelian group containing $\Ga$, such that $R_{\ga} \cdot
M_{\delta} \subseteq M_{\ga + \delta}$ for all $\ga \in \Ga_R, 
\delta \in \Ga'$. 
Then, $\Gamma_M$ and $M^h$ are
defined analogously to $\Gamma_R$ and~$R^h$.  We say that $M$ is 
a {\it graded free} $R$-module if it has a base as a free
$R$-module consisting of homogeneous elements.

A graded ring $E = \bigoplus_{ \ga \in \Gamma} E_{\ga}$ is
called a \emph{graded division ring} if $\Ga$ is a 
torsion-free abelian group and 
every non-zero homogeneous
element of $E$ has a multiplicative inverse. 
Note that the grade set  $\Ga_{E}$ is actually a group.  
Also, $E_{0}$ is a division ring,
and  $E_\gamma$ is a $1$-dimensional 
left and right $E_0$ vector space for every $\gamma\in \Gamma_E$.
The requirement that $\Gamma$ be torsion-free is made
because we are interested in graded division rings arising
from valuations on division rings, and all the grade groups
appearing there are torsion-free.  Recall that every 
torsion-free abelian group $\Gamma$ admits total orderings compatible 
with the group structure.  (For example, $\Gamma$ embeds in 
$\Gamma \otimes _{\mathbb Z}\mathbb Q$ which can be given 
a lexicographic total ordering using any base of it as a 
$\mathbb Q$-vector space.)  By using any total ordering on 
$\Gamma_E$, it is easy to see that $E$ has no zero divisors
and that $E^*$, the multiplicative group of units of $E$, 
coincides with  
 $E^{h} \mi \{0\}$ (cf. \cite{hwcor}, p.~78). 
Furthermore, the degree map 
\begin{equation}\label{degmap}
\deg\colon E^* \rightarrow \Gamma_E
\end{equation} is a group homomorphism with kernel $E_0^*$. 


By an easy adaptation of the ungraded arguments, one can see
that every graded  module~$M$  over a graded division ring 
$E$ is graded free, and every two  
homogenous bases have the same cardinality. 
We thus call $M$ a \emph{graded vector space} over $E$ and
write $\dim_E(M)$ for the rank of~$M$ as a graded free $E$-module.
Let $S \subseteq E$ be a graded subring which is also a graded
division ring.  Then, we can view $E$ as a graded left $S$-vector 
space, and  we write $[E:S]$ for $\dim_S(E)$.  It is easy to 
check the \lq\lq Fundamental Equality,"
\begin{equation}\label{fundeq}
[E:S] \ = \ [E_0:S_0] \, |\Gamma_E:\Gamma_S|,
\end{equation}
where $[E_0:S_0]$ is the dimension of $E_0$ as a left vector space
over the division ring $S_0$ and $|\Gamma_E:\Gamma_S|$ denotes the 
index in the group $\Gamma_E$ of its subgroup $\Gamma_S$.

A \emph{graded field} $T$ is a commutative graded division ring.
Such a $T$ is an integral domain, so it has a quotient field, 
which we denote $q(T)$.  It is known, see \cite{hwalg}, Cor.~1.3, 
that $T$ is integrally closed in $q(T)$.  An extensive theory of
graded algebraic extensions of graded fields has been developed in 
\cite{hwalg}.
For a graded field $T$, we can define a
grading on the polynomial ring $T[x]$ as follows:  
Let  $\Delta$ be a totally ordered abelian group with
$\Gamma_T \subseteq \Delta$, and fix
$\theta \in \Delta$.  We have 
\begin{equation}\label{homogenizable} 
T[x]  \,  = \,   
\textstyle\bigoplus\limits_{\ga \in 
\Delta} T[x]_{\ga}, \quad
\mathrm{where}\quad
T[x]_{\ga} \  = \  \{ \textstyle\sum a_{i} x^{i} \mid a_{i} \in T^{h}, \
\deg(a_{i}) +i \theta = \ga \}. 
\end{equation}
This makes $T[x]$ a graded ring, which we denote $T[x]^{\theta}$.
Note that $\Gamma_{T[x]^{\theta}} = \Gamma_T +
\langle \theta \rangle$.
 A homogeneous polynomial in $T[x]^{\theta}$
 is said to be
{\it $\theta$-homogenizable}.
If $E$ is a graded division 
algebra with center $T$, and
$a \in E^{h}$ is homogeneous of
degree $\theta$, then the evaluation homomorphism 
$\epsilon_a\colon  T[x]^{\theta} \ra T[a]$
given by  $f \mapsto f(a)$ is a graded ring 
homomorphism.  Assuming $[T[a]:T]<\infty$, we 
have $\ker(\epsilon_a)$ is a principal ideal of 
$T[x]$ whose unique monic generator $h_a$ is called the 
minimal polynomial of~$a$ over $T$. It is known, see \cite{hwalg}, 
Prop.~2.2, that
if $\deg(a) = \theta$, then $h_a$ is $\theta$-homogenizable. 

If $E$ is a graded division ring, then its center $Z(E)$ is clearly
a graded field.  {\it The graded division rings considered in 
this paper will always be assumed finite-dimensional over their 
centers.}  The finite-dimensionality assures that $E$ 
has a quotient division ring $q(E)$ obtained by central localization,
i.e., $q(E) = E \otimes_T q(T)$ where $T = Z(E)$. Clearly,
$Z(q(E)) = q(T)$ and $\ind(E) = \ind(q(E))$, where  
the index of $E$ is defined by $\ind(E)^2 = [E:T]$.  
If $S$ is a graded field which is a graded subring of $Z(E)$
and $[E:S] <\infty$, 
then $E$ is said to be a {\it graded division algebra} over~$S$.

A graded division algebra $E$ with center $T$ is 
said to be {\it unramified} if $\Gamma_E=\Gamma_T$. 
From~(\ref{fundeq}), it follows then that $[E:S]=[E_0:T_0]$. 
At the other extreme, $E$ is said to be {\it totally ramified} 
if $E_0=T_0$. In a case in the middle, $E$ is 
said to be {\it semiramified} if  
 $E_0$ is a field and ${[E_0:T_0]=|\Gamma_E:\Gamma_T|=\ind(E)}$. 
These definitions are motivated by  analogous  definitions for 
valued division algebras (\cite{wadval}). 
Indeed,  if a valued division algebra 
is unramified, semiramified, or totally ramfied, then so is its 
associated graded division 
algebra (see \S\ref{valued}).

A main theme of this paper is to study the correspondence 
between  $\SK$ of a valued division 
algebra and that of  its associated graded division algebra.  
We now recall how to associate a graded division algebra to a valued 
division algebra.

Let $D$ be a division algebra finite dimensional over its 
center $F$, with a valuation 
$v: D^{\ast} \ra
\Ga$. So $\Ga$ is a totally ordered abelian group,  
and $v$ satisifies the conditions that for all
$a, b \in D^{\ast}$,
\begin{enumerate}
\item $ v(ab) = v(a) + v(b)$;

\item $v(a+b) \geq \min \{v(a),v(b) \}\;\;\;\;\; (b \neq -a).$
\end{enumerate}
Let 
\begin{align*}
V_D  \ &= \  \{ a \in D^{\ast} : v(a) \geq 0 \}\cup\{0\}, 
\text{ the
valuation ring of $v$};\\ 
M_D  \ &= \  \{ a \in D^{\ast} : v(a) > 0
\}\cup\{0\}, \text{ the unique maximal left (and right) ideal
 of $V_D$}; \\
\overline{D}  \ &= \  V_D / M_D, \text{ the residue
division ring of $v$ on $D$; and} \\
\Ga_D  \ &= \  \mathrm{im}(v), \text{ the value
group of the valuation}. 
\end{align*}
For background on  valued division algebras, 
see \cite{jw} or the survey paper  \cite{wadval}.  
One associates to $D$ a graded division algebra 
as follows:
For each $\gamma\in \Gamma_D$, let
\begin{align*} 
 D^{\ge\ga}  \ &=  \ 
\{ d \in D^{\ast} : v(d) \geq \ga \}\cup\{0\}, \text{ an additive 
subgroup of $D$}; \qquad \qquad\qquad\qquad\qquad \ \\
D^{>\ga}  \ &=  \ \{ d \in D^{\ast} : v(d) > \ga \}\cup\{0\}, 
\text{ a subgroup 
of $D^{\ge\ga}$};   \text{ and}\\
 \gr(D)_\gamma \ &= \ 
D^{\ge\ga}\big/D^{>\ga}. 
\end{align*}
Then define
$$
 \gr(D)  \ = \  \textstyle\bigoplus\limits_{\ga \in \Ga_D} 
\gr(D)_\gamma. \ \ 
$$
Because $D^{>\ga}D^{\ge\de} \,+\, D^{\ge\ga}D^{>\de} 
\subseteq D^{>(\ga +
\de)}$ for all $\ga , \de \in \Ga_D$, the  multiplication on 
$\gr(D)$ induced by multiplication on $D$ is
well-defined, giving that $\gr(D)$ is a graded  ring, called the 
{\it associated graded ring} of $D$. The 
multiplicative property 
(1) of  the valuation $v$ implies that $\gr(D)$ is a graded 
division ring.
Clearly,
we have ${\gr(D)}_0 = \overline{D}$ and $\Ga_{\gr(D)} = \Ga_D$.
For $d\in D^*$, we write $\widetilde d$ for the image 
$d + D^{>v(d)}$ of $d$ in $\gr(D)_{v(d)}$.  Thus, 
the map given by $d\mapsto \widetilde d$ is 
a group epimorphism $ D^* \rightarrow {\gr(D)^*}$ with 
kernel~$1+M_D$.  

The restriction $v|_F$ of the valuation on $D$ to its center $F$, is
a valuation on $F$, which induces a corresponding graded field  $\gr(F)$.
Then it is clear that $\gr(D)$ is a graded $\gr(F)$-algebra, and 
by (\ref{fundeq})  and the Fundamental Inequality for 
valued division algebras, 
$$
[\gr(D):\gr(F)]  \ = \ 
[\overline{D}:\overline{F}] \, |\Ga_D :\Ga_F|
 \ \le  \ [D:F] \ < \infty.
$$

Let $F$ be a field with a henselian valuation $v$. Recall that a 
 field extension $L$ of $F$ of degree~$n<\infty$ is 
said to be {\it tamely ramified} or {\it tame} over $F$ 
if, with respect to the unique extension of $v$ to $L$,
the residue field 
$\overline L$ is a separable field extension of~ 
$\overline F$ and $\chr({\overline F})\nmid
n\big/[\overline L:\overline F]$. Such an $L$~is 
necessarily {\it defectless} over $F$, 
i.e., ${[L:F] = [\overline L:\overline F] \, |\Gamma_L:\Gamma_F|
=[\gr(L):\gr(F)]}$.
Along the same lines, let $D$ be a
division algebra with center $F$ (so, by convention,
$[D:F] < \infty$); then $v$ on $F$ extends uniquely to a valuation 
on $D$. With respect to this valuation, $D$ 
is said to be  
{\it tamely ramified} or {\it tame} if $Z(\overline D)$ is 
 separable over $\overline F$ and ${\chr({\overline F}) \nmid 
\ind(D)\big/\big(\ind(\overline D)[Z(\overline D):\overline F]\big)}$. 
It is known (cf.~Prop.~4.3 in \cite{hwcor}) that 
$D$ is tame if and only if $[\gr(D):\gr(F)] = [D:F]$ and 
$Z(\gr(D))=\gr(F)$, if and only if $D$ is split by the 
maximal tamely ramified extension of~$F$, if and only if 
$\chr(\ov F) = 0$
or $\chr(\ov F) = p\ne 0$ and the $p$-primary component of 
$D$ is split by the maximal unramified extension of $F$.
We say $D$ is \emph{strongly tame} 
if $\chr(\overline F)\nmid\ind(D)$.
Note that strong tameness implies tameness. 
This is clear from the last characterization of tameness, 
or from \eqref{Ostrowski} below.
For a detailed 
study of the associated graded algebra of a valued 
division algebra
refer to \S4 in \cite{hwcor}. Recall also from \cite{M},
Th.~3, that for a valued 
division algebra $D$ finite dimensional over its  center $F$ (here not necessarily 
henselian), we have the \lq\lq Ostrowski theorem"
\begin{equation}\label{Ostrowski}
[D:F]
\ = \ q^k\,[\overline D:\overline F] \,|\Gamma_D:\Gamma_F|
\end{equation} 
where $q=\chr({\overline D})$ and $k \in \mathbb Z$ with $k\geq 0$
(and  $q^k = 1$ if $\chr(\overline D) = 0$).  
If $q^k = 1$ in equation~   \eqref{Ostrowski}, then $D$ is 
said to be  {\it defectless} over $F$.

Let $E$ be a graded division algebra with, as we always assume,
 $\Gamma_E$ a 
torsion-free abelian group.
After fixing  some total ordering on $\Ga_E$,  define a function
$$ 
\la\colon  E\mi\{0\} \ra E^{\ast} \quad\textrm{ by } 
\quad\la(\textstyle\sum c_{\ga}) =
c_{\de},
$$ 
where $\de$ is minimal among the $ \ga \in \Ga_E$ with
$c_{\ga} \neq0$. Note that $\la(a) = a$ for $a \in E^{\ast}$, and
\begin{equation}\label{valhomin}
 \la(ab)= \la(a) \la(b) \textrm{ for all } a,b \in E\mi\{0\}.
\end{equation}

Let $Q = q(E)$.
We can extend $\la$ to a map defined on all of $Q^{\ast}$
 as follows: for $q
\in Q^{\ast}$, write $q = ac^{-1}$  with $a \in E\mi\{0\}$, $c \in
Z(E)\mi\{0\}$, and set $\la(q) = \la(a) \la(c)^{-1}$. It follows
 from (\ref{valhomin}) that $\la\colon Q^{\ast} \ra E^{\ast}$ is 
well-defined and is
a group homomorphism. Since the composition $E^{\ast}
\hookrightarrow Q^{\ast} \ra E^{\ast}$ is the identity, $\la$ is a
splitting map for the injection $E^{\ast} \hookrightarrow 
Q^{\ast}$. (In  Lemma~\ref{skeskqinj} below, we will observe that 
this 
map induces a monomorphism from $\SK(E)$ to $\SK(Q)$.)

Now, by composing $\la$ with the degree map of (\ref{degmap}) we get a 
map $v$,
\begin{equation}\label{lamfun}
\begin{split}
\xymatrix{
Q^* \ar[r]^{\la} \ar[dr]^v & E^*\ar[d]^{\deg} \\
& \Gamma_E}
\end{split}
\end{equation}
This  $v$ is in fact a valuation on $Q$: for $a,b \in Q^*$, 
$v(ab)=v(a)+v(b)$ as $v$ is the composition of two group 
homomorphisms, and it is straightforward to check that  
$v(a+b) \ge \min(v(a),v(b))$ (check this first for $a,b \in 
E  \mi \{0\}$).  It is easy to see that 
for the associated graded ring for this valuation 
on $q(E)$, we have $\gr(q(E))\conggr E$;
this is a strong indication of the close connection between
graded and valued structures.

\section{Reduced norm and reduced Whitehead 
group of a graded division algebra} \label{reducedsec}

Let $A$ be an Azumaya algebra of constant rank $n^2$
over a commutative ring $R$. 
Then there is a  commutative 
ring $S$ faithfully flat over $R$ which splits $A$, i.e., 
$A\otimes_R S \cong M_n(S)$. For $a \in A$, 
considering $a \otimes 1$ as an element of $M_n(S)$, 
one then defines the {\it reduced characteristic polynomial}, 
 the {\it reduced trace}, and the {\it reduced 
norm} of $a$ by
  $$
\chr_A(x,a) \ = \ \det(x-(a\otimes1)) \ = \ 
x^n-\Trd_A(a)x^{n-1}+\ldots+(-1)^n\Nrd_A(a).
$$
Using descent theory, one shows that $\chr_A(x,a)$ is independent of $S$ 
and of the choice of isomorphism $A\otimes_RS \cong M_n(S)$, and 
that $\chr_A(x,a)$ lies in $R[x]$;  
furthermore,  the element $a$ is invertible in $A$ if and only if 
$\Nrd_A(a)$ is invertible in $R$ (see Knus \cite{knus}, III.1.2, 
and Saltman \cite{saltman}, Th.~4.3). Let $A^{(1)}$ denote the set 
of elements of $A$ with the reduced 
norm $1$. One then  defines the {\it reduced Whitehead group} 
of $A$ to be $\SK(A)=A^{(1)}/A'$, where $A'$ denotes the commutator 
subgroup of the group $A^*$ of  invertible elements of $A$.  The {\it reduced norm residue group} of $A$ is defined to be 
$\SH(A)=R^*/\Nrd_A(A^*)$.  These groups are related by the exact 
sequence: 
$$
1\longrightarrow \SK(A) \longrightarrow A^*/A' 
\stackrel{\Nrd}{\longrightarrow} R^*  \longrightarrow \SH(A) 
\longrightarrow 1
$$

Now let $E$ be a graded division algebra with 
center $T$. Since $E$ is an Azumaya  algebra over~ 
$T$ (\cite{boulag}, Prop.~5.1 or\cite{hwcor}, Cor.~1.2), its 
reduced Whitehead group $\SK(E)$ is defined.   

\begin{rem}
The reduced norm for an Azumaya algebra 
is defined using a splitting ring, and in general
splitting rings can be difficult to find. But   
for a graded division algebra $E$ we observe that,
analogously to the  case of ungraded division rings,  
any maximal graded subfield $L$ of $E$ splits $E$. 
For, the centralizer $C = C_E(L)$ is a graded subring of
$E$ containing $L$, and for any homogeneous $c\in 
C$, $L[c]$ is a graded subfield of $E$ containing $L$.
Hence, $C = L$, showing that $L$ is a maximal commutative 
subring of $E$. Thus,  by Lemma~5.1.13(1), p.~141 of \cite{knus}, 
as  $E$ is Azumaya,  $E \otimes_T L \cong \End_L(E) \cong M_n(L)$.
Thus, we can compute reduced norms for elements of $E$ by 
passage to $E\otimes_T L$.
\end{rem}

 We have 
other tools as well for computing $\Nrd_E$
and~$\Trd_E$:

\begin{proposition}\label{normfacts}
Let $E$ be a graded division ring with center $T$.
Let $q(T)$ be the quotient field of 
$T$, and let ${q(E) = E\otimes _T q(T)}$, which is the 
quotient division ring of  $E$.
We view $E\subseteq q(E)$.
Let $n = \ind(E) = \ind(q(E))$.
Then for any $a\in E$,
\begin{enumerate}
\item [(i)] $\chr_E(x,a) = \chr_{q(E)}(x,a)$, 
so 
\begin{equation}\label{nrdquot}
\Nrd_E(a)  \, =  \, \Nrd_{q(E)}(a)
 \quad\text{ and }
\quad \Trd_E(a)  \, = \,  \Trd_{q(E)}(a).
\end{equation}
\item [(ii)] If $K$ is any graded subfield of $E$
containing $T$ and $a\in K$, then 
$$
\Nrd_E(a)  \,= \, N_{K/T}(a)^{n/[K:T]}
\quad \text{ and }\quad \Trd_E(a) \, =  \, \textstyle
\frac n{[K:T]}\, \Tr_{K/T}(a).
$$
\item [(iii)] For $\gamma\in \Gamma_E$, if 
$a\in E_\gamma$ then $\Nrd_E(a)\in E_{n\gamma}$ and 
$\Trd(a)\in E_\gamma$.  In particular, $E^{(1)} 
\subseteq E_0$.
\item [(iv)] Set $\dlambda = \ind(E)\big/\big(\ind(E_0)
[Z(E_0):T_0]$\big).  If $a\in E_0$, then, 
\begin{equation}\label{NrdD0}
\Nrd_E(a) \, = \, N_{Z(E_0)/T_0}\Nrd_{E_0}(a)^
 {\,\dlambda} 
\in T_0
\quad \text{ and }\quad
\Trd_E(a) \, = \, \dlambda \,
\Tr_{Z(E_0)/T_0}\Trd_{E_0}(a) \in T_0.
\end{equation}
\end{enumerate}  
\end{proposition}

\begin{proof}
(i) The construction of reduced characteristic polynonials 
described above is clearly compatible with scalar extension
of the ground ring.   
Hence, $\chr_E(x,a) = \chr_{q(E)}(x,a)$ (as we are identifying
$a\in E$ with 
$a\otimes1$ in $E\otimes _T q(T)$\,).  
The formulas in \eqref{nrdquot} follow immediately.

(ii) Let $h_a = x^m  +t_{m-1}x^{m-1} + \ldots+ t_0 \in q(T)[x]$ 
be the minimal 
polynomial of $a$ over~$q(T)$.  As noted in \cite{hwalg}, Prop.~2.2, 
since the integral domain $T$ is integrally closed and $E$~is 
integral over~$T$, we have $h_a\in T[x]$.  Let 
$\ell_a = x^k + s_{k-1}x^{k-1} + \ldots +s_0\in T[x]$ be the 
characteristic polynomial of the $T$-linear function on the 
free $T$-module $K$ given by  $c\mapsto ac$.  By definition,
$N_{K/T}(a) = (-1)^ks_0$ and $\Tr_{K/T}(a) = -s_{k-1}$.
Since
$q(K) = K\otimes_T q(T)$, we have $[q(K):q(T)] = [K:T] = k$
and $\ell_a$ is also the characteristic polynomial for the 
$q(T)$-linear transformation of $q(K)$ given by $q \mapsto aq$.
So, $\ell_a = h_a^{k/m}$.  Since 
$\chr_{q(E)}(x,a) = h_a^{n/m}$ (see \cite{rei}, Ex.~1, p.~124), 
we have 
$\chr_{q(E)}(x,a) = \ell_a^{n/k}$. Therefore, using (i),
$$
\Nrd_E(a)  \ =  \ \Nrd_{q(E)}(a)  \ = \ \big[(-1)^ks_0\big]^{n/k}
  \ = \ N_{K/T}(a)^{n/k}.
$$
The formula for $\Trd_E(a)$ in (ii) follows analogously.

(iii) From the equalities
$\chr_E(x,a) = \chr_{q(E)}(x,a) = h_a^{n/m}$ noted in 
proving (i) and (ii), we have $\Nrd_E(a) = [(-1)^mt_0]^{n/m}$
and $\Trd_E(a) = -\frac nm t_{m-1}$.
As noted in \cite{hwalg}, Prop.~2.2, if $a\in E_\gamma$, 
then its minimal polynomial $h_a$ is $\gamma$-homogenizable
in $T[x]$ as in \eqref{homogenizable} above.  Hence,
$t_0\in E_{m\gamma}$ and $t_{m-1}\in E_\gamma$.  Therefore,
$\Nrd_E(a) \in E_{n\gamma}$ and $\Trd(a) \in E_{\gamma}$.  
If $a \in E^{(1)}$ then $a$ is homogeneous, since it is a 
unit of $E$, and since $1 = \Nrd_E(a) \in E_{n\deg(a)}$,
necessarily $\deg(a) = 0$.

(iv) Suppose $a\in E_0$.  Then, $h_a$ is $0$-homogenizable 
in $T[x]$, i.e., $h_a\in T_0[x]$.  Hence, 
$h_a$  is the minimal polynomial of 
$a$ over the field $T_0$.  Therefore, if $L$ is any maximal subfield 
of $E_0$ containing $a$, we have 
$N_{L/T_0}(a) = [(-1)^mt_0]^{[L:T_0]/m}$. Now, 
 $$
n/m  \ = \  \delta \ind(E_0)[Z(E_0):T_0]\big/m \ = 
\  \delta\,[L:T_0]/m.
$$ 
Hence,
\begin{align*}
\Nrd_E(a) \ &= \ \big[(-1)^m t_0\big]^{n/m} \ = \ 
\big[(-1)^m t_0\big]^{\delta[L:T_0]/m} \ = \ N_{L/T_0}(a)^\de
 \\
&= \ N_{Z(E_0)/T_0}N_{L/T_0}(a)^\de \ = \ 
N_{Z(E_0)/T_0}\Nrd_{E_0}(a)^\de.
\end{align*}
The formula for $\Trd_E(a)$ is proved analogously.
\end{proof}

In the rest of this section we study the reduced Whitehead group 
$\SK$ of a graded division algebra. As we mentioned 
in the introduction, the motif is to show that working 
in the graded setting is much easier than in the 
non-graded setting.

The most successful approach to computing $\SK$ for 
 division algebras over henselian fields is due to Ershov 
in \cite{ershov}, where three linked exact 
sequences were constructed involving a division algebra $D$, its 
residue division algebra $\ov D$,  and its group of
 units $U_D$ (see also \cite{wadval},
 p.~425).  From these exact sequences, Ershov recovered  Platonov's 
 examples \cite{platonov} of division algebras with nontrivial 
$\SK$ and many more examples as well.
In this section we will easily  prove the graded version of
 Ershov's exact sequences (see diagram~\eqref{ershovsk1}),
 yielding formulas for $\SK$ of 
unramified,  semiramified, and totally ramified 
graded  division algebras. This will be applied in 
\S\ref{valued}, where it will be  shown that $\SK$ of 
a tame  division algebra over a henselian field 
coincides with $\SK$ of 
its associated graded division algebra. We can then
readily deduce from the graded results
many   established 
formulas in the literature for the reduced Whitehead 
groups of  valued division algebras   
(see Cor. ~\ref{sk1appl}).  This demonstrates the merit of the 
graded approach.

If $N$ is a group, we denote by $N^n$ the subgroup of $N$
generated by all $n$-th powers of elements of $N$. A
homogeneous multiplicative commutator of $E$ where $E$ is
a graded division ring, is an element of the form $ab a^{-1}
b^{-1}$ where $a,b \in E^{*} = E^h \mi \{0\}$.
 We will  use the notation $[a,b]
= ab a^{-1} b^{-1}$ for $a,b \in E^{h}$. 
Since $a$ and $b$ are homogeneous, note that $[a,b] \in E_0$.
If $H$ and $K$ are
subsets of $E^{\ast}$, then $[H, K]$ denotes the subgroup of 
$E^{\ast}$
generated by ${\{[h,k] :h \in H, k \in
K\}}$. The group $[E^{\ast}, E^{\ast}]$ will be denoted by $E'$.

\begin{proposition} \label{normalthm} 
Let $E = \bigoplus_{\al \in \Ga} E_{\al}$ be a graded division
algebra with graded center $T$, with $\ind(E) = n$. Then,
\begin{enumerate}
\item [(i)] If $N$ is a normal subgroup of 
$E^{\ast}$, then $N^{n} \subseteq
\Nrd_E(N)[E^{\ast}, N]$.
\item [(ii)] $\SK(E)$ is $n$-torsion.
\end{enumerate}
\end{proposition}

\begin{proof}
Let $a \in N$ and let $h_a \in q(T)[x]$ be the minimal polynomial of
$a$ over $q(T)$, and let $m = \deg(h_a)$.
  As noted in the proof of Prop.~\ref{normfacts}, 
$h_a \in T[x]$ and $\Nrd_E(a) =[(-1)^mh_a(0)]^{n/m}$.
 By the graded
Wedderburn Factorization Theorem~\ref{domainW}, we have
$h_a = (x - d_1 a d_{1}^{-1}) \ldots (x - d_{m} a d_{m}^{-1})$
where each $d_{i} \in E^{*}\subseteq E^h$. Note that 
$[E^*,N]$ is a normal subgroup of $E^*$, since $N$ is normal 
in $E^*$.
It follows that
\begin{align*} \Nrd_{E}(a)  \ &= \  
\big(d_{1} a d_{1}^{-1} \ldots d_{m} a d_{m} ^{-1}\big)^{n/m}  \ = \ 
 \big([d_1 , a] a [d_2 , a] a \ldots a [d_m , a] a\big)^{n/m} \\
&= a^n d_a \;\;\;\; \mathrm{where} \;\; d_a \in [E^{\ast} , N]. 
\end{align*}
Therefore, $a^n = \Nrd_{E}(a) d_a^{-1}\in \Nrd_E(N)[E^*,N]$,
yielding (i). (ii)  is immediate  from (i) by taking $N = E^{(1)}$.
\end{proof}

The fact that $\SK(E)$ is $n$-torsion is also deducible from
the injectivity of the map $\SK(E) \to \SK(q(E))$ shown in 
Lemma~\ref{skeskqinj} below.


We  recall the definition of the group 
$\widehat H^{-1}(G,A)$, which will appear in 
our description of $\SK(E)$. For any finite
group $G$ and any $G$-module $A$,
 define the norm map 
$N_G:A\rightarrow A$ as follows: for any $a \in A$, let
$N_G(a)=\sum_{g \in G}ga$. Consider the $G$-module 
$I_G(A)$ generated as an abelian group by 
$\{a-ga: a \in A \text{ and }g \in G\}$. 
Clearly, $I_G(A) \subseteq \ker (N_G)$. Then, 
\begin{equation}\label{Hminus1}
\widehat H^{-1}(G,A) \ = \ \ker (N_G)\big/I_G(A).
\end{equation}

\begin{theorem} \label{bigdiag}
Let $E$ be any graded division ring finite 
dimensional over its center $T$. Let $\dlambda =
 \ind(E)\big/\big(\ind(E_0)\,
[Z(E_0):T_0]\big)$, and let $\mu_\dlambda(T_0)$ be 
the group of those
$\dlambda$-th roots of unity lying in $T_0$.
Let $G = \Gal(Z(E_0)/T_0)$ and let
$\widetilde N = N_{Z(E_0)/T_0}\circ \Nrd_{E_0}\colon
E_0^* \to T_0^*$.
Then, the rows and column of the following 
diagram are exact:
\begin{equation}\label{ershovsk1}
\begin{split}
\xymatrix{
 & & 1 \ar[d] & \\
 & \SK(E_0) \ar[r] & \ker \widetilde N/[E_0^*,E^*] 
\ar[r]^-{\Nrd_{E_0}} \ar[d] & \widehat H^{-1}(G,\Nrd_{E_0}(E_0^*)) 
\ar[r] &  1 \\
& 
\Gamma_E\big/\Gamma_T \wedge \Gamma_E\big/\Gamma_T 
\ar[r] & E^{(1)}/[E_0^*,E^*] 
\ar[r]  \ar[d]^{\widetilde N}& \SK(E) \ar[r] & 1\\
 & & \mu_\dlambda(T_0) \cap  \widetilde N (E_0^*) \ar[d]&\\
 & & 1 &}
 \end{split}
 \end{equation}
\end{theorem}

\begin{proof} 
By Prop.~2.3 in \cite{hwcor},  
$Z(E_0)/T_0$ is a Galois extension and the map 
$\theta\colon E^* \rightarrow \Aut(E_0)$, 
given by $e \mapsto (a\mapsto eae^{-1})$ 
for $a \in E_0$, induces an epimorphism 
$E^* \rightarrow G =\Gal(Z(E_0)/T_0)$.  
In the notation for \eqref{Hminus1} with $A
=\Nrd_{E_0}(E_0^*)$, we have $N_G$ 
coincides with $N_{Z(E_0)/T_0}$ on $A$.  
Hence, 
\begin{equation}\label{kerNG}
\ker(N_G) \ = \ \Nrd_{E_0}(\ker(\widetilde N)).
\end{equation}
Take any $e\in E^*$ and let $\sigma = \theta(e)
\in \Aut(E_0)$. 
For any $a\in E_0^*$, let $h_a \in Z(T_0)[x]$ be the minimal
polynomial of $a$ over $Z(T_0)$. Then $\sigma(h_a)\in Z(T_0)[x]$
is the minimal polynomial of $\sigma(a)$ over $Z(T_0)$.
Hence, $\Nrd_{E_0}(\sigma(a)) =\sigma(\Nrd_{E_0}(a))$. 
Since $\sigma|_{Z(T_0)} \in G$, this yields 
\begin{equation}\label{Nrdbracket}
\Nrd_{E_0}([a,e]) \ =  \ \Nrd_{E_0}(a\sigma(a^{-1})) \ = \ 
\Nrd_{E_0}(a) \sigma(\Nrd_{E_0}(a))^{-1} \ \in \ I_G(A),
\end{equation} 
hence $\widetilde N([a,e]) = 1$.  Thus, we have 
$[E_0^*, E^*]\subseteq \ker(\widetilde N) \subseteq E^{(1)}$
with the latter inclusion from Prop.~\ref{normfacts}(iv).  
The formula in Prop.~\ref{normfacts}(iv) also shows that 
$\widetilde N(E^{(1)}) \subseteq \mu_\dlambda(T_0)$.  Thus, 
the vertical maps in diagram~\eqref{ershovsk1} are well-defined, and 
the column in \eqref{ershovsk1} is exact.  Because 
$\Nrd_{E_0}$ maps $\ker(\widetilde N)$ onto $\ker(N_G)$ by 
\eqref{kerNG} and it maps $[E_0^*,E^*]$ onto $I_G(A)$ by 
\eqref{Nrdbracket} (as $\theta(E^*)$ maps onto $G$), the map
labelled $\Nrd_{E_0}$ in diagram~\eqref{ershovsk1} is surjective with 
kernel $E_0^{(1)}\,[E_0^*,E^*]\big/[E_0^*, E^*]$.  Therefore, the
top row of \eqref{ershovsk1} is exact.  For the lower row, 
since $[E^*,E^*] \subseteq E_0^*$ and $E^*\big/(E_0^*\, Z(E^*))
\cong \Gamma_E /\Gamma_T$, the following lemma yields an epimorphism
$\Gamma_E /\Gamma_T \wedge \Gamma_E /\Gamma_T \to 
[E^*,E^*]/[E_0^*,E^*]$. Given this, the
lower row in \eqref{ershovsk1} is evidently exact.
\end{proof}

\begin{lemma}\label{wedge}
Let $G$ be a group, and let $H$ be a subgroup of $G$ with $H\supseteq
[G,G]$.  Let ${B = G\big/(H\, Z(G))}$.  Then, there is an epimorphism 
$B\wedge B \to [G,G]\big/[H,G]$.
\end{lemma}

\begin{proof}
Since $[G,G]\subseteq H$, we have 
$\big[[G,G],[G,G]\big] \subseteq [H,G]$, so $[G,H]$ is a normal 
subgroup of $[G,G]$ with abelian factor group.  Consider the map
$\beta\colon G\times G \to [G,G]/[H,G]$ given by 
$(a,b)\mapsto aba^{-1}b^{-1}[H,G]$.  For any $a,b,c \in G$ we have 
the commutator identity ${[a,bc] = [a,b]\, [b,[a,c]]\, [a,c]}$.
The middle term $[b,[a,c]]$ lies in $[H,G]$.  Thus, $\beta$ is 
multiplicative in the second variable; likewise, it is multiplicative 
in the first variable.  As ${[ H\,Z(G), G]\subseteq [H,G]}$, this
$\beta$ induces a well-defined group homomorphism 
$\beta'\colon B\otimes_{\mathbb Z} B \to [G,G]/[H,G]$, which is 
surjective since $\im(\beta)$ generates $[G,G]/[H,G]$.  Since 
$\beta'(\eta\otimes \eta) = 1$ for all $\eta\in B$, 
there is an induced epimorphism $B\wedge B\to [G,G]/[H,G]$.
\end{proof}

\begin{corollary} \label{skunramthm}
Let $E$ be a graded division ring with  
graded center $T$.

\begin{enumerate}
\item [(i)] If $E$ is unramified, then $\SK(E) \cong \SK(E_0)$.

\item [(ii)]  If $E$ is totally ramified, then
$\SK(E)\cong\mu_n(T_0)/\mu_e(T_0)$ where $n = \ind(E)$ and 
$e$ is the exponent of $\Gamma_E/\Gamma_T$.

\item[(iii)] If $E$ is semiramified, then for 
$G = \Gal(E_0/T_0) \cong 
\Gamma_E/\Gamma_T$ there is an exact sequence
\begin{equation}\label{semiramseq1}
G\textstyle\wedge G  \ \to  \  \widehat H^{-1}(G,E_0^*) 
 \ \to \  \SK(E) \  \to \  1.
\end{equation}

\item[(iv)] If $E$ has  maximal subfields $L$ and $K$ which are 
respectively unramified and totally ramified 
over $T$, then $E$ is semiramified and 
$\SK(E)\cong \widehat H^{-1}(\Gal(E_0/T_0),E_0^*)$.
\end{enumerate}
\end{corollary}

\begin{proof}  (i)   Since $E$ is unramified over $T$, we have
$E_0$ is a 
central $T_0$-division algebra, ${\ind(E_0)=\ind(E)}$, and 
$E^*=E_0^*T^*$. It follows that $G=\Gal(Z(E_0)/T_0)$ is trivial,
 and thus 
$\widehat H^{-1}(G,\Nrd_{E_0}(E_0))$ is trivial;
also, $\dlambda=1$,  
and from (\ref{NrdD0}), $\Nrd_{E_0}(a)=\Nrd_E(a)$ for all $a\in E_0$.
Furthermore, $[E_0^*, E^*] = [E_0^*, E_0^*T^*] = [E_0^*, E_0^*]$ 
as $T^*$ is central. 
Plugging this information  into  the exact top row of diagram  
(\ref{ershovsk1}) and 
noting that the exact sequence extends to the left by 
$1\rightarrow [E_0^*,E^*]/[E_0^*,E_0^*]\rightarrow \SK(E_0)$, 
part (i) follows.

 (ii) When $E$ is totally ramified, $E_0=T_0$, $\dlambda=n$,   
$\widetilde N$ is  the 
identity map on $T_0$, and ${[E^*, E_0^*] = [E^*,T_0^*] = 1}$.  
Plugging all this into the exact column  of 
diagram (\ref{ershovsk1}), it follows that 
$E^{(1)}\cong \mu_n(T_0)$. Also by \cite{hwcor} Prop.~2.1,
$E'\cong\mu_e(T_0)$ where $e$ is the exponent of 
the torsion abelian group $\Gamma_E/ \Gamma_T$.  
Part (ii) now follows.

 (iii)  As recalled at the beginning of the proof of 
Th.~\ref{bigdiag}, for any graded division algebra~$E$ with
center $T$, we have $Z(E_0)$ is Galois over $T_0$, and there
is an epimorphism ${\theta\colon E^* \to \Gal(Z(E_0)/T_0)}$.
Clearly, $E_0^*$ and $T^*$ lie in $\ker(\theta)$, so 
$\theta$ induces an epimorphism $\theta'\colon \Gamma_E/\Gamma_T \to 
\Gal(Z(E_0)/T_0)$.  When $E$ is semiramified, by definition 
${[E_0:T_0] = |\Gamma_E:\Gamma_T| = \ind(E)}$ and $E_0$ is a field.
Let $G = \Gal(E_0/T_0)$. Because $|G| = [E_0:T_0] = 
|\Gamma_E:\Gamma_T|$, the map $\theta'$ must be an isomorphism.
In diagram~\eqref{ershovsk1}, since $\SK(E_0) = 1$ and clearly 
$\delta = 1$, the exact top row and column yield 
$E^{(1)}\big/[E_0^*, E^*] \cong \widehat H^{-1}(G, E_0^*)$.
Therefore, the exact row~\eqref{semiramseq1} follows from the 
exact second row of diagram~\eqref{ershovsk1} and the 
isomorphism $\Gamma_E/\Gamma_T \cong G$ given by $\theta'$.

 (iv) Since $L$ and $K$ are maximal subfields 
of $E$, we have ${\ind(E) = [L:T] = [L_0:T_0] \le [E_0:T_0]}$
and $\ind(E) = [K:T] = |\Gamma_K:\Gamma_T| \le |\Gamma_E:\Gamma_T|$.  
It follows from \eqref{fundeq} that these inequalities are
equalities, so $E_0 = L_0$ and $\Gamma_E = \Gamma_K$.  
Hence, $E$ is semiramified, and (iii) applies.
  Take any $\eta, \nu\in \Gamma_E/\Gamma_T$, and any inverse images 
$a,b$ of $\eta, \nu$ in $E^*$.  The left map in \eqref{semiramseq1}
sends $\eta \wedge \nu$ to $aba^{-1} b^{-1}$ mod $I_G(E_0^*)$.
Since $\Gamma_E = \Gamma_K$, these $a$ and $b$ can be chosen in $K^*$,
so they commute.  Thus, the left map of \eqref{semiramseq1} is trivial 
here, yielding the isomorphism of (iv).
\end{proof}

For a graded division algebra  $E$  with center $T$, define 
\begin{equation}\label{ckdef}
\CK(E)  \,=  \, E^{\ast}\big/ (T^{\ast} E').
\end{equation}
This is the graded analogue to $\CK(D)$ for a division algebra $D$,
which is defined as 
\break$\CK(D) = D^*\big/(F^*D')$, where $F = Z(D)$.
That is, $\CK(D)$ is the cokernel of the canonical 
map $K_1(F) \to K_1(D)$.  See \cite{sk12001} for background on  
$\CK(D)$. Notably, it is known that $\CK(D)$ is 
torsion of bounded exponent $n=\ind(D)$, and $\CK$ has  functorial 
properties similar to~$\SK$. 
The $\CK$ functor was used in \cite{hazwadsworth} in showing that
for \lq\lq nearly all" 
division algebras~$D$, the multiplicative group 
$D^*$ has a maximal proper subgroup. 
It is conjectured (see \cite{hazwadsworth} and its references)
that  if $\CK(D)$ is trivial, then 
 $D$ is a quaternion division 
algebra (necessarily over  a real Pythagorean field).


Now, for the graded division algebra $E$ with center $T$, the 
degree map (\ref{degmap})  induces a surjective map 
$E^{\ast} \ra \Ga_{E} / \Ga_{T}$ which has kernel 
$T^{\ast}
{E_{0}}^{\ast}$. One can then observe that there is an 
exact sequence
\begin{displaymath}
1  \ \lra  \ {{E_{0}}^{\ast}}\big/{{T_{0}}^{\ast} E'} 
 \ \lra  \ \CK(E)  \ \lra \ 
{\Gamma_{E}}\big/{\Gamma_{T}}  \ \lra \  1.
\end{displaymath}
Thus if $E$ is unramified,  
$\CK(E) \cong {E_0}^{\ast} /({T_{0}}^{\ast}
E')$ and  $E^{\ast} \cong T^{\ast} {E_{0}}^{\ast}$. 
It then  follows that
$E' \cong {E_{0}}'$, yielding $\CK(E) \cong \CK(E_0)$. At the other extreme, when $E$ is
totally ramified then ${E_{0}}^{\ast} \big/ ({T_{0}}^{\ast} E')
= 1$, so the exact sequence 
above yields $\CK(E) \cong \Ga_E /\Ga_T$.

\section{$\SK$ of a valued division algebra and its associated graded 
division algebra}\label{valued}

 The aim of this section is to study the relation between 
the reduced Whitehead group (and other related functors) of a 
valued division algebra with that of its corresponding graded division 
algebra. We will prove that $\SK$ of a tame valued division 
algebra over a henselian field
 coincides with $\SK$ of its associated graded division 
algebra.  We start by recalling the concept of 
$\lambda$-polynomials introduced in \cite{MWad}.  
We keep the notations introduced 
in \S\ref{prel}.

Let $F$ be a field with valuation $v$, 
let $\gr(F)$ be the associated graded
field, and  $F^{alg}$ the algebraic closure of $F$. For 
$a \in F^{\ast}$,
let $\ti{a} \in \gr(F)_{v(a)}$ be the image of $a$ in $\gr(F)$, 
let $\ti{0}=
0_{\gr(F)}$, and  for $f = \sum a_i x^i \in F[x]$, let 
$\ti{f} = \sum {\ti{a}}_i x^i
\in \gr(F)[x]$. 

\begin{deff}\label{defnum}
Take any $\la$ in the divisible hull of $\Ga_F$ and let 
 ${f= a_n x^n
+ \ldots +a_i x^i + \ldots + a_0 \in F[x]}$ with $a_n a_0 \neq 0$. 
Take any extension of $v$ to $F^{alg}$. We
say that $f$ is a {\it $\la$-polynomial} if it satisfies the following
equivalent conditions:
\begin{enumerate}
\item[(a)] Every root of $f$ in $F^{alg}$ has value $\la$;

\item[(b)] $v(a_i) \geq (n-i)\la +v(a_n)$ for all $i$ and 
$v(a_0) = n \la +v(a_n)$;

\item[(c)] Take  any
$c \in F^{alg}$ with $v(c) = \la$ and let ${h=
\frac{1}{a_n c^n} f(cx) \in F^{alg}[x]}$; then $h$ is monic in
$V_{F^{alg}}[x]$ and $h(0) \neq 0$ (so $h$ is a $0$-polynomial).
\end{enumerate}
\end{deff}

If $f$ is a $\la$-polynomial, let 
\begin{equation}\label{flambda}
f^{(\la)}  \ = \ \textstyle \sum\limits_{i=0}^{n}
a'_i x^i  \, \in \,  \gr(F)[x],
\end{equation}
where $a'_i$ is the image of $a_i$ in
$\gr(F)_{(n-i)\la + v(a_n)}$ (so $a_0' = \widetilde {a_0}$
and $a_n' = \widetilde{a_n}$, but  for $1\le i \le n-1$,
$a'_i =0$ if $v(a_i) > (n-i) \la
+v(a_n)\,$).
Note that $f^{(\la)}$ is a homogenizable polynomial in $\gr(F)[x]$, i.e.,
$f^{(\la)}$ is homogeneous (of degree $v(a_0)$) with respect
to the  the grading on  $\gr(F)[x]$ as in \eqref{homogenizable} with 
$\theta = \lambda$. Also, $f^{(\la)}$ has the same degree
as $f$as a polynomial in $x$.
 
The $\lambda$-polynomials are useful generalizations of polynomials
$h\in V_F[x]$ with $h(0)\ne 0$---these are $0$-polynomials.  The 
following proposition collects some basic properties of 
$\lambda$-polynomials over henselian fields, which are 
analogous to well-known 
properties for $0$-polynomials, and have similar proofs. 
See, e.g., \cite{EP}, Th.~4.1.3, pp.~87--88 for proofs for 
$0$-polynomials, and 
\cite{MWad} for proofs for $\lambda$-polynomials.

\begin{proposition}  \label{lambdafacts} 
Suppose the valuation $v$ on $F$ is henselian.
Then,
\begin{enumerate}
\item[(i)] If $f$ is a $\la$-polynomial and $f=gh$ in $F[x]$, then $g$
and $h$ are $\la$-polynomials and $f^{(\la)}= g^{(\la)} h^{(\la)}$
in $\mathrm{gr}(F)[x]$. So, if $f^{(\la)}$ is irreducible in 
$\mathrm{gr}(F)[x]$, 
then $f$ is irreducible in $F[x]$.

\item[(ii)] If $f= \sum_{i=0}^{n} a_i x^i$ is irreducible in $F[x]$ with
$a_n a_0 \neq 0$, then 
$f$ is a
$\la$-polynomial for $\la = (v(a_0)-v(a_n))/n$. Furthermore,
$f^{(\la)} = \ti{a_n} h^s$ for some irreducible monic
$\la$-homogenizable polynomial $h \in \mathrm{gr}(F)[x]$.

\item[(iii)] If $f$ is a $\la$-polynomial in $F[x]$ and 
if $f^{(\la)} = g' h'$ in $\mathrm{gr}(F)[x]$ with 
$\mathrm{gcd}(g', h')=1$, then there exist $\la$-polynomials 
$g, h \in F[x]$ such that $f=gh$  and
$g^{(\la)} = g'$ and $h^{(\la)}=h'$.

\item[(iv)] If $f$ is a $\la$-polynomial in $F[x]$ and if 
$f^{(\la)}$ has
a simple root $b$ in $\mathrm{gr}(F)$, then $f$ has a simple root $a$ in $F$
with $\ti a=b$.

\item[(v)] Suppose $k$ is a $\la$-polynomial in $\gr(F)[x]$ with 
$k(0)\neq 0$, and suppose $f \in F[x]$ with $\ti f =k$. Then $f$ is a
$\la$-polynomial and $f^{(\la)}=k$.

\end{enumerate}
\end{proposition}

\begin{lemma}\label{grnorm}
Let $F \subseteq K$ be fields with $[K:F] < \infty$. Let $v$ be a
henselian valuation on $F$ such that $K$ is defectless over $F$.
Then, for every $a \in K^{\ast}$, with $\widetilde{a}$ its image in
${\mathrm{gr}(K)}^{\ast}$,
$$
\widetilde{N_{K/F} (a)}  \ = \  
N_{\mathrm{gr}(K) /\mathrm{gr}(F)} (\widetilde{a}).
$$
\end{lemma}
\begin{proof}
Let $n=[K:F]$. Note that $[\gr(K):\gr(F)] = n$ as $K$ is defectless
over $F$. Let 
\break$f= x^{\ell} +c_{\ell-1} x^{\ell -1} +
\ldots + c_0 \in F[x]$ be the minimal polynomial of $a$ over $F$.
Then $f$ is irreducible in $F[x]$ and since $v$ is henselian, 
 $f$ is a $\la
$-polynomial, where $\la = v(a) = v(c_0)/n$ (see 
Prop.~\ref{lambdafacts}(ii)). Let $f^{(\la)}$ be
the corresponding $\la$-homogenizable polynomial in $\gr(F)[x]$
as in \eqref{flambda}.
Then $f^{(\la)}(\widetilde{a})=0$ in $\gr(K)$
(by Prop.~\ref{lambdafacts}(i) with $g = x-a$), and 
by Prop.~\ref{lambdafacts}(ii) $f^{(\la)}$ has  
only one monic irreducible factor in $\gr(F)[x]$, say 
$f^{(\la)} =h^s$, with $\deg(h) = \ell/s$.
Since $f^{(\la)}(\widetilde{a}) =0$, $h$ must be the minimal
polynomial of $\widetilde{a}$ over $\gr(F)$ and over 
$q(\gr(F))$. (Recall that since $\gr(F)$ is integrally closed,
a monic polynomial in $\gr(F)[x]$ is irreducible in $\gr(F)[x]$
iff it is irreducible in $q(\gr(F))[x]$.)  
We have $N_{K/F}(a)   = {(-1)}^{n}c_0^{n/\ell}$. 
Hence, as $q(\gr(K)) \cong \gr(K)  \otimes_{\gr(F)} q(\gr(F))$, 
\begin{align*}
 N_{\gr(K)/\gr(F)}(\ti{a} )  \ &=  \  N_{q(\gr(K))/q(\gr(F))}(\ti a)
 \ = \ (-1)^n h(0)^{n s / \ell}  \ =
 \  (-1)^n {(h(0)^s)}^{n/\ell}
  \\ &= \  (-1)^n (\widetilde{c_0}^{n /
\ell})  \ = \  \widetilde{(-1)^n {c_0}^{n/\ell}}  \ 
= \  \widetilde{N_{K/F}(a)}.\qedhere
\end{align*}
\end{proof}

\begin{remark} The preceding lemma is still valid if $v$ on $F$ is not assumed
to be henselian, but merely assumed to have a unique  and defectless
extension to $K$. This can be proved by scalar extension to the
henselization $F^h$ of $F$. (Since $v$
extends uniquely and defectlessly to $K$,
$K \otimes_F F^h$ is a field,  and 
${\gr(K \otimes_F F^h) \conggr \gr(K)}$.)

\end{remark}

\begin{cor} \label{nrdcor}
Let $F$ be a field with henselian valuation $v$, and let $D$ be a
tame $F$-central division algebra. Then for every $a \in
D^{\ast}$, $\Nrd_{\mathrm{gr}(D)}(\ti{a}) = \widetilde{\Nrd_D(a)}$.
\end{cor}
\begin{proof}  Recall from \S\ref{prel} that the assumption $D$ is 
tame over $F$ means
that ${[D:F]=[\gr(D):\gr(F)]}$ and $\gr(F) = Z(\gr(D))$. Take any maximal
subfield $L$ of $D$ containing $a$. Then $L/F$ is defectless as
$D/F$ is defectless, 
so $[\gr(L):\gr(F)] = [L:F] = \ind(D) = \ind(\gr(D))$.
  Hence,  using  Lemma~\ref{grnorm} and Prop.~\ref{normfacts}(ii), 
we have,
\[ 
\widetilde{\Nrd_D(a)}  \ = \  \widetilde{N_{L/F}(a)} 
 \ = \  N_{\gr(L)/\gr(F)}(\ti{a})  \ = \  \Nrd_{\gr(D)}(\ti{a}). \qedhere
\]
\end{proof}

\begin{remarks}
 (i) Again, we do not need that  $v$ be henselian for 
Cor.~\ref{nrdcor}. It suffices that
the valuation $v$ on $F$ extends to $D$ and $D$ is tame over $F$.

(ii) Analogous results hold for the trace and reduced trace, with
analogous proof. In the setting of Lemma~\ref{grnorm}, we have: if
$v(\Tr_{K/F}(a)) = v(a)$, then $\Tr_{K/F}(\ti{a}) =
\widetilde{\Tr_{K/F}(a)}$, but if $v(\Tr_{K/F}(a)) > v(a)$, then
$\Tr_{K/F}(\ti{a}) = 0$.

(iii)
By combining Cor.~\ref{nrdcor} with  equation~(\ref{NrdD0}), 
for a tame valued division algebra $D$ over henselian field $F$, 
we can relate the reduced norm of 
$D$ with the reduced norm of $\overline D$ as follows:
\begin{equation}\label{ershov}
\overline{\Nrd_D(a)} \ = \ N_{Z(\overline D)/\overline F}
\Nrd_{\overline D}(\overline
a)^\dlambda,
\end{equation}\label{ershnrd}
for any $a \in V_D \backslash  M_D$ 
(thus, $\Nrd_D(a) \in V_F \backslash M_F$) and 
$\dlambda=
\ind(D)\big/\big(\ind(\overline D)\,
[Z(\overline D):\overline F]\big)$  (cf.~\cite{ershov},~Cor.~2).
\end{remarks}

The next proposition will be used several times below.  
It was proved by Ershov in \cite{ershov}, Prop. 2,  
who refers to {Yanchevski\u\i} \cite{y} for part of the  
argument.  We give a proof here for the convenience 
of the reader, and also to illustrate the utility of
$\lambda$-polynomials.

\begin{prop}\label{normsurj}
Let $F \subseteq K$ be fields with henselian valuations
 $v$ such that $[K:F]<\infty$ and $K$ is tamely ramified over $F$. Then $N_{K/F}(1+M_K)=1+M_F$.
\end{prop}
\begin{proof} If $s \in 1+M_K$ then $\ti{s}=1$ in $\gr(K)$. So, as 
$K$ is defectless over
 $F$ by Lemma.~\ref{grnorm},  
$\ti{N_{K/F}(s)}=N_{\gr(K)/\gr(F)}(\ti{s})=1$ in $\gr(F)$, i.e., 
$N_{K/F}(s) \in 1+M_F$. Thus $N_{K/F}(1+M_K) \subseteq1+M_F$.\\
To prove that this inclusion is an equality, we can assume $[K:F]>1$.
We have 
\break 
$[\gr(K):\gr(F)]=[K:F]>1$, since tamely ramified extensions
are defectless.   
Also, the tame ramification implies that  
$q(\gr(K))$ is separable over $q(\gr(F))$.  For, $q(\gr(F)) \cdot
\gr(K)_0$ is separable over $q(\gr(F))$ since 
$\gr(K)_0 = \ov K$ and $\ov K$ is separable over $\gr(F)_0 = \ov F$.
But also, $q(\gr(K))$ is separable over $q(\gr(F)) \cdot \gr(K)_0$
because $[q(\gr(K)):q(\gr(F) \cdot \gr(K)_0] = |\Gamma_K:\Gamma_F|$,
which is not a multiple of $\chr(\ov F)$.
 Now, take any homogenous element $b \in \gr(K)$, 
$b\not \in \gr(F)$, and let $g$ be the minimal polynomial of $b$ 
over $q(\gr(F))$. Then $g \in \gr(F)[x]$, $b$ is a simple root of 
$g$, and $g$ is $\lambda$-homogenizable where $\lambda=\deg(b)$,
by \cite{hwalg}, Prop.~2.2. 
Take any monic $\lambda$-polynomial $f \in F[x]$ with 
$f^{(\lambda)}=g$. Since $f^{(\lambda)}$ has the simple root $b$ 
in $\gr(K)$ and the valuation on $K$ is henselian, 
by Prop.~\ref{lambdafacts}(iv) there is 
$ a\in K$ such that $a$ is a simple root of~$f$
and $\widetilde a = b$. Let 
$L=F(a) \subseteq K$. Write $f=x^n+c_{n-1}x^{n-1}+\ldots+c_0$. 
Take any $t \in 1+M_F$, and let $h=x^n+c_{n-1}x^{n-1}+\ldots+ c_1x+
tc_0 \in F[x]$. 
Then $h$ is a $\lambda$-polynomial (because $f$ is) and 
$h^{(\lambda)}=f^{(\lambda)} = g$ in $\gr(F)[x]$. Since $h^{(\lambda)}$ 
has the simple root $b$ in $\gr(L)$, $h$ has a simple root ~$d$ in 
$L$ with $\ti{d}= b = \ti{a}$
by Prop.~\ref{lambdafacts}(iv). So, $da^{-1} \in 1+M_L$. 
The polynomials $f$ and $h$ are irreducible $F[x]$ by 
Prop.~\ref{lambdafacts}(i), as $g$ is irreducible in 
$\gr(F)[x]$.  Since $f$ 
(resp.~$h$), is the minimal polynomial of $a$ (resp.~$d$) over $F$, 
we have $N_{L/F}(a)=(-1)^n c_0$ and $N_{L/F}(d)=(-1)^n c_0t$. Thus, 
$N_{L/F}(da^{-1})=t$, showing 
that $N_{L/F}(1+M_L)=1+M_F$. If $L=K$, we are done. If not, we 
have $[K:L] <[K:F]$, and $K$ is tamely ramified over $L$. So, by 
induction on $[K:F]$, we have $N_{K/L}(1+M_K)=1+M_L$. 
Hence, 
\[
N_{K/F}(1+M_K) \ = \ N_{L/F}\big(N_{K/L}(1+M_K)\big)
 \ = \ N_{L/F}(1+M_L) \ = \ 1+M_F.
\qedhere
\]
\end{proof}


\begin{corollary} \label{Dnormsurj}
Let $F$ be a field with henselian valuation $v$, and let $D$ be an 
$F$-central 
division algebra which is tame with respect to $v$. Then, 
$\Nrd_D(1+M_D)=1+M_F$. 
\end{corollary}
\begin{proof} Take any $a\in 1+M_D$ and any maximal subfield $K$ of 
$D$ with $a \in K$. Then, $K$ is defectless over $F$, since $D$ is 
defectless over $F$. So, $a \in 1+M_K$, and 
${\Nrd_D(a)=N_{K/F}(a) \in 1+M_F}$ by the first part of the proof of  
Prop.~\ref{normsurj}, which required only defectlessness, not 
tameness. Thus, $\Nrd_D(1+M_D) \subseteq 1+M_F$. For the reverse 
inclusion, recall from \cite{hwcor}, Prop.~4.3 that as $D$ is tame over $F$, 
it has a maximal 
subfield $L$ with $L$ tamely ramified over $F$. Then by 
Prop.~\ref{normsurj}, 
$$
1+M_F \ = \ N_{L/F}(1+M_L) \ = \ \Nrd_D(1+M_L) \ \subseteq  \ 
\Nrd_D(1+M_D) \ \subseteq  \ 1+M_F,
$$ 
so equality holds throughout.
\end{proof}

We can now prove the main result of this section:
\begin{thm}\label{sk1prop}
Let $F$ be a field with henselian valuation $v$ and let $D$ be a
tame $F$-central division algebra. Then 
$\SK (D) \cong \SK (\mathrm{gr}(D))$.
\end{thm}
\begin{proof}
Consider the canonical surjective group homomorphism 
$\rho\colon D^{\ast} \ra
\gr(D)^{\ast}$ 
given by $a \mapsto \ti{a}$. 
Clearly, $\ker(\rho) = 1+M_D$.  If 
$a\in D^{(1)} \subseteq V_D$ then $\ti a \in \gr(D)_0$ and
by Cor.~\ref{nrdcor}, 
$$
\Nrd_{\gr(D)}(\ti a) \ = \ \ti{\Nrd_D(a)} \ = \,1.
$$
This shows that $\rho (D^{(1)}) \subseteq \gr(D)^{(1)}$. Now 
consider the diagram
\begin{equation} \label{diagram1}
\begin{split}
\xymatrix{1 \ar[r] & (1+M_D) \cap D' \ar[r] \ar[d] & D' \ar[d]
\ar[r]^-{\rho} & \gr(D)' \ar[d] \ar[r] & 1 \\
1 \ar[r] & (1+M_D)\cap D^{(1)} \ar[r] & D^{(1)} \ar[r] &
\gr(D)^{(1)} \ar@{.>}[r] & 1}
\end{split}
\end{equation}
The top row of the above diagram  is clearly exact. The Congruence 
Theorem (see Th.~\ref{congru} in Appendix~\ref{apencongru}),
implies that the left
vertical map in the diagram  is an isomorphism. Once we prove
that $\rho(D^{(1)}) = \gr(D)^{(1)}$, we will have the exactness of
the second row of diagram (\ref{diagram1}), and the theorem 
follows by the exact sequence for cokernels.

To prove the needed surjectivity, take any $b \in \gr(D)^{\ast}$
with $\Nrd_{\gr(D)}(b) =1$. Thus $b \in \gr(D)_0$ by Th.
~\ref{normalthm}. Choose $a
\in V_D$ such that $\ti a=b$. Then we have,
$$
\overline{\Nrd_D(a)} \ = \ \ti{\Nrd_D(a)} \ 
= \ \Nrd_{\gr(D)}(b) \ = \, 1.
$$    
Thus $\Nrd_D(a)
\in 1+M_F$. By Cor.~\ref{Dnormsurj}, since $\Nrd_D(1+M_D)=1+M_F$, 
there is $c \in 1+M_D$ such that $\Nrd_D(c) = \Nrd(a)^{-1}$.  Then,
$ac\in D^{(1)}$ and $\rho(ac) =\rho(a) = b$.
\end{proof}

 Recall from \S\ref{prel} that starting from any graded division
algebra $E$ with center $T$ and any choice of total ordering
$\le$ on the torsion-free abelian group $\Gamma_E$, there is an 
induced valuation~$v$ on $q(E)$, see \eqref{lamfun}.   Let $h(T)$
be the henselization of $T$ with respect to $v$, and let 
${h(E) = q(E) \otimes_{q(T)} h(T)}$.  Then, $h(E)$ is a division 
ring by Morandi's henselization theorem (\cite{M}, Th.~2 or see
\cite{wadval}, Th.~ 2.3), and with respect to the unique extension of 
the henselian valuation on $h(T)$ to $h(E)$, $h(E)$ is an immediate 
extension $q(E)$, i.e., $\gr(h(E)) \conggr \gr(q(E))$.  Furthermore,
as
$$
[h(E):h(T)] \ = \ [q(E):q(T)] \ = \ [E:T] \ = \ 
[\gr(q(E)) :\gr(q(T))] \ = \ [\gr(h(E):\gr(h(T))]
$$
and 
$$
Z(\gr(h(E))) \  \conggr \, Z(\gr(q(E)))  \ \conggr \, T \  \conggr
\, \gr(h(T)) \ = \ 
\gr(Z(h(E))), 
$$
$h(E)$ is tame (see the characterizations of tameness in \S\ref{prel}).

\begin{cor}
Let $E$ be a graded division algebra. Then $\SK(h(E)) \cong
\SK(E)$.
\end{cor}
\begin{proof}
Since $h(E)$ is a tame valued division algebra, 
by Th.~ \ref{sk1prop}, $\SK(h(E)) \cong \SK(\gr(h(E)))$. But 
$\gr(h(E)) \conggr\gr(q(E))\conggr E$, so the corollary follows.
\end{proof}

 Having now
established that the reduced Whitehead group of a division algebra 
coincides with that of its associated graded division algebra, 
we can easily  deduce
stability of $\SK$ for unramified valued division algebra, due 
originally to 
Platonov (Cor.~3.13 in \cite{platonov}), and also a formula for $\SK$ for a 
totally ramified division algebra (\cite{lewistig}, p.~363, see 
also \cite{ershov}, p.~70), and also a formula for $\SK$ in the nicely 
semiramfied case (\cite{ershov}, p.~69), as  natural consequences 
of Th.~\ref{sk1prop}:

\begin{cor}\label{sk1appl}
Let $F$ be a field with Henselian valuation, and let $D$ be a 
tame division algebra with center $F$.

\begin{enumerate}
\item [(i)] If $D$ is unramified then $\SK(D) \cong \SK(\overline D)$

\item [(ii)] If $D$ is totally ramified then 
$\SK(D)\cong\mu_n(\overline F)/\mu_e(\overline F)$ 
where $n = \ind(D)$ and $e$ is the exponent of 
$\Gamma_D /\Gamma_F$.

\item [(iii)]  If $D$ is semiramified, let 
$G = \Gal(\ov D/\ov F) \cong \Gamma_D/\Gamma_F$. Then, there is 
an exact sequence
\begin{equation*}
G\wedge G  \ \to  \  \widehat H^{-1}(G,\ov D^*) 
 \ \to \  \SK(D) \  \to \  1.
\end{equation*}

\item [(iv)] If $D$ is nicely semiramfied, then $\SK(D)\cong
\widehat H^{-1}(\Gal(\overline D/\overline F),\overline D^*)$.
\end{enumerate}
\end{cor}

\begin{proof}  Because $D$ is tame, $Z(\gr(D)) = \gr(F)$
and $\ind(\gr(D)) = \ind(D)$.  Therefore, for $D$ in each case 
(i)--(iv) here, $\gr(D)$ is in the corresponding case of 
Cor.~\ref{skunramthm}.  (In case (iii), that $D$ is semiramified 
means  $[\ov D:\ov F] = |\Gamma_D:\Gamma_F| = \ind(D)$
and $\ov D$ is a field.  Hence $\gr(D) $ is semiramified.  
In case (iv),  since  
$D$ is nicely semiramified, by definition (see \cite{jw}, p.~149)
it contains maximal subfields $K$ and 
$L$, with $K$ unramified over $F$ and $L$ totally ramified
over~$F$.  (In fact, by \cite{mou}, Th.~2.4, $D$ is nicely 
semiramified if and only if it has such maximal subfields.) 
Then, $\gr(K)$ and $\gr(L)$ are maximal graded subfields
of $\gr(D)$ by dimension count and the graded double centralizer 
theorem,\cite{hwcor}, Prop.~1.5(b),
with $\gr(K)$~unramified over~$\gr(F)$ and $\gr(L)$ 
totally ramified over $\gr(F)$. So, $\gr(D)$ is  then in case~(iv)
of Cor.~\ref{skunramthm}.)
Thus, in each case Cor.~\ref{sk1appl} for $D$ follows from 
Cor.~\ref{skunramthm} for $\gr(D)$ together with the isomorphism 
$\SK(D) \cong \SK(\gr(D))$ given by Th.~\ref{sk1prop}.
\end{proof}

Recall that the reduced norm residue group of $D$ is defined as
$\SH(D)=F^*/\Nrd_D(D^*)$. It is known that $\SH(D)$ coincides with
the first Galois cohomology group $H^1(F,D^{(1)})$ (see
\cite{kmrt}, \S29). We now show that for a tame
division algebra $D$ over a henselian field, $\SH(D)$ coincides with $\SH$ of its
associated graded division algebra.

\begin{thm}\label{shthm}
Let $F$ be a field with a henselian valuation $v$ and let $D$ be a
tame $F$-central division algebra. Then 
$\SH (D) \cong \SH (\mathrm{gr}(D))$.
\end{thm}
\begin{proof}
Consider the diagram with exact rows,
\begin{equation} \label{diagram}
\begin{split}
\xymatrix{1 \ar[r] & 1+M_D  \ar[r] \ar[d] & D^* \ar[d]^{\Nrd_D}
\ar[r]^-{\rho} & \gr(D)^* \ar[d]^{\Nrd_{\gr(D)}} \ar[r] & 1 \\
1 \ar[r] & 1+M_F \ar[r] & F^* \ar[r] & \gr(F)^* \ar[r] & 1}
\end{split}
\end{equation}
where Cor. ~\ref{nrdcor} guarantees that the diagram is commutative.
By Cor.~\ref{Dnormsurj},  the left vertical map is 
an epimorphism.
The theorem follows by the exact sequence for cokernels.
\end{proof}

\begin{remark}As with $\SK$, if $D$ is tame and unramified, then 
$$
\SH(D) \ \cong \ 
\SH(\gr(D))  \ \cong \  \SH(\gr(D)_0)  \ \cong  \ \SH (\overline D).
$$
\end{remark}
  
We conclude this section by establishing a similar result for the 
$\CK$ functor of \eqref{ckdef} above. Note that here, unlike the situation with $\SK$ 
(Th. ~\ref{sk1prop}) or with $\SH$ (Th.~\ref{shthm}), we need to assume  
strong tameness here. 

\begin{thm} \label{ckunramtramthm}
Let $F$ be a field with henselian valuation $v$ and let $D$ be a
strongly tame $F$-central division algebra. Then 
$\CK(D) \cong \CK(\mathrm{gr}(D))$. 
\end{thm}

\begin{proof} Consider the canonical epimorphism 
$\rho\colon D^* \rightarrow {\gr(D)^*}$ 
given by $a\mapsto \widetilde a$, with kernel $1+M_D$. 
Since $\rho$ maps $D'$ onto $\gr(D)'$ and $F^*$ onto
$\gr(F)^*$, it induces an isomorphism 
 $D^*\big/\big(F^*D'(1+M_D) \big)\cong 
\gr(D)^*\big/\big(\gr(F)^*\gr(D)'\big)$. We have
$\gr(F) = Z(\gr(D))$ and by
Lemma~2.1 in \cite{haz}, as $D$ is strongly tame,  
$1+M_D=(1+M_F)[D^*,1+M_D] \subseteq F^* D'
$.  Thus,
$\CK(D) \cong \CK(\gr(D))$.
\end{proof}

\section{Stability of the reduced Whitehead group} \label{skqd}

The goal of this section is to prove that if $E$ is a graded division
ring (with $\Ga_E$ a torsion-free abelian group), then $\SK(E) \cong
\SK(q(E))$, where $q(E)$ is the quotient division ring of $E$. When
$\Ga_E \cong \mathbb{Z}$, this was essentially proved by Platonov
and {Yanchevski\u\i} in \cite{py}, Th.~1 (see the Introduction). Their argument was based on properties of twisted
polynomial rings, and our argument is based on their approach. So,
we will first look at twisted polynomial rings. For these, an
excellent reference is Ch.~1 in \cite{j}.


 Let $D$ be a division ring
finite dimensional over its center $Z(D)$. Let $\si$ be an
automorphism of $D$ whose restriction to $Z(D)$ has finite order,
say $\ell$. Let $T=D[x,\si]$ be the twisted polynomial ring, with
multiplication given by $xd= \si(d) x$, for all $d \in D$. By Skolem-Noether,
there is $w \in D^{\ast}$ with $\si^{\ell} = \intt (w^{-1})$ (=
conjugation by $w^{-1}$); moreover,  $w$ can be chosen so that
$\si(w) = w$ (by a Hilbert 90 argument, see 
\cite{j}, Th.~1.1.22(iii) or \cite{py}, Lemma~1).
Then $Z(T) = K[y]$ (a commutative polynomial ring), where $K=
Z(D)^{\si}$, the fixed field of $Z(D)$ under the action of $\si$,
 and
$y=wx^{\ell}$. Let $Q=q(T)=D(x, \si)$, the division ring of
quotients of $T$. Note that $Z(Q)=q(Z(T)) =K(y)$, and $\ind(Q) =
\ell \ind(D)$. Observe that within $Q$ we have the twisted Laurent
polynomial ring $T[x^{-1}]= D[x, x^{-1}, \si]$ which is a graded
division ring, graded by degree in $x$, and $T \subseteq T[x^{-1}]
\subseteq q(T)$, so that $q(T[x^{-1}]) = Q$. Recall that, since we
have left and right division algorithms for $T$, $T$ is a principal
left (and right) ideal domain.

Let $\mathcal{S}$ denote the set of isomorphism classes $[S]$ of simple
left $T$-modules $S$, and set 
$$
\Div(T) \ = \  \textstyle\bigoplus\limits_{[S] \in \mas}
\mathbb{Z}[S],
$$
the free abelian group with base $\mas$. For any
$T$-module $M$ satisfying both ACC and DCC, the Jordan-H\"older
Theorem yields a well-defined element $jh(M) \in \Div(T)$, given by
$$
jh(M)  \, = \ \textstyle \sum\limits_{[S] \in \mas} n_{[S]} (M)[S],
$$
 where $n_{[S]}
(M)$ is the number of appearances of simple factor modules isomorphic to
$S$ in any composition series of $M$.
Note that for any $f \in T\mi \{0\}$, the division algorithm shows
that ${\dim_D(T / Tf) =\deg(f) < \infty}$. Hence, $T/Tf$ has ACC and
DCC as a $T$-module. Therefore, we can define a divisor function
$$
 \de\colon T\setminus \{0\} \ra \Div(T),  \ \ \textrm{ given by } 
 \ \ \de(f) \, = \, jh(T/Tf).
$$

\begin{rem}\label{deremarks} Note the following properties of $\de$:
\begin{enumerate}
\item [(i)] For any $f,g \in T\setminus \{0\}$, $\de(fg) = \de(f) +\de(g).$ This
follows from the isomorphism $Tg/ Tfg \cong T/ Tf$ (as $T$ has no
zero divisors).

\item [(ii)] We can extend $\de$ to a map 
$\de\colon Q^{\ast}\ra \Div(T)$, 
where $Q=q(T)$,  by $\de (fh^{-1}) = \de(f) -
\de(h)$ for any $f \in T\setminus \{0\}$, $h \in Z(T) \setminus \{0\}$. It follows
from (i) that $\de$ is well-defined and is a group homomorphism on
$Q^{\ast}$. Clearly, $\delta$ is surjective, as every simple
$T$-module is cyclic.

\item [(iii)] For all $q, s \in Q^{\ast}$, $\de(sqs^{-1}) = \de(q)$. This is
clear, as $\de$ is a homomorphism into an abelian group.

\item [(iv)] \label{note4} For all $q \in Q^{\ast}$, $\de(\Nrd_Q(q))= n \,
\de(q)$, where
$n= \ind(Q)$. This follows from (iii), since 
Wedderburn's factorization theorem applied to the minimal
polynomial of $q$ over $Z(Q)$ shows    
that $\Nrd_Q(q) = \prod_{i=1}^{n} s_iq
{s_i}^{-1}$ for some $s_i \in Q^{\ast}$.

\item [(v)] If $\Nrd_Q(q)= 1$, then $\de(q) =0$. This is immediate from
(iv), as $\Div(T)$ is torsion-free.
\end{enumerate}
\end{rem} 

\begin{lemma} \label{fs=tg}
Take any $f,g \in T\mi\{0\}$ with $T/Tf \cong T/Tg$, so $\deg(f) =
\deg(g)$. If $\deg(f) \geq 1$, there exist $s, t \in T
\setminus\{0\}$ with
$\deg(s) = \deg(t) < \deg(f)$ such that $fs=tg$.
\end{lemma}

\begin{proof} (cf.~\cite{j}, Prop.~1.2.8) We have $\deg(f)=
\dim_D(T/Tf) = \dim_D(T/Tg) = \deg(g)$. Let $\al\colon T/Tf \ra T/Tg$ be
a $T$-module isomorphism, and let $\al(1+Tf)= s +Tg$. By the
division algorithm, $s$ can be chosen with $\deg(s)<\deg(g)$. We
have
\begin{equation*}
fs+Tg  \ = \  f(s+Tg)  \ = \  f \al(1+Tf)   \ 
= \   \al (f+ Tf)  \ = \  \al(0)  \ 
= \  0 \mathrm{\;\;\;\; in \;\;} T/Tg.
\end{equation*}
Hence, $fs=tg$ for some $t \in T$. Since $\deg(f) = \deg(g)$, we
have 
\[ 
\deg(t)  \, = \,  \deg(s)  \, < \,  \deg(g)  \, = \,  \deg(f).\qedhere 
\]
\end{proof}

\begin{prop} \label{kerde}  Consider the group homomorphism
$\de\colon Q^* \rightarrow \Div(T)$ defined in 
Remark~\ref{deremarks}$($ii$)$ above. Then  $\ker(\de) = D^*Q'$.
\end{prop}
\begin{proof} (cf.~\cite{py}, proof of Lemma~5) Clearly, $D^{\ast}
\subseteq \ker(\de)$ and $Q' \subseteq \ker(\de)$, so $D^*Q' \subseteq \ker(\de)$. For the reverse inclusion
take $h \in \ker(\de)$ and write $h = f/g$ with $f,g \in T
\setminus\{0\}$.
Since $\de(f/g)=0$, we have $\de(f) = \de(g)$, so $\deg(f) =
\deg(g)$. If $\deg(f)=0$, then $h \in D^{\ast}$, and we're done. So,
assume $\deg(f) >1$. Write $f= pf_1$ with $p$ irreducible. Then,
$T/Tp$ is one of the simple composition factors of $T/Tf$. If 
$g=q_1q_2 \ldots q_k$ with each $q_i$ irreducible, then the composition
factors of $T/Tg$ are (up to isomorphism) $T/Tq_1, \ldots, T/Tq_k$.
Because $\de(f) = \de(g)$, i.e. $jh(T/Tf)= jh(T/Tg)$, we must have
$T/Tp \cong T/Tq_j$ for some $j$. Write $g= g_1q
g_2$ where $q= q_j$. By Lemma ~\ref{fs=tg}, there exist $s,t \in
T\setminus
\{0\}$ with $\deg(s) = \deg(t) < \deg(p) = \deg(q)$ and $ps=tq$.
Then, working modulo $Q'$, we have
\begin{equation*}
 h  \, =  \, fg^{-1} \, = \  (pf_1)(g_1 q g_2)^{-1}  \, \equiv \
f_1(pq^{-1})(g_1g_2)^{-1} \, 
  \equiv \ f_1(ts^{-1})(g_1g_2)^{-1}  \ \equiv \ 
(f_1t)(g_1g_2s)^{-1}.
\end{equation*}

Let $h'=(f_1t)(g_1g_2s)^{-1}$. Since $h' \equiv h \ (\text{mod }  Q')$, we have
$\de(h') = \de(h) = 0$, while ${\deg(f_1t) < \deg(f)}$. By iterating
this process we can repeatedly lower the degree of numerator and
denominator to obtain $h'' \in D^{\ast}$ with $h'' \equiv h' \equiv
h \ (\text{mod }Q')$. Hence, $h \in D^*Q'$, as
desired.  \end{proof}

\begin{remark}  Since $K_1(Q) = Q^*/Q'$, 
Prop.~\ref{kerde} can be stated as saying that there is an
exact sequence
\begin{equation}\label{locseq}
\K(D) \longrightarrow \K(Q) \stackrel{\de}{\longrightarrow} \Div(T)
\longrightarrow 0.
\end{equation}
This can be viewed as part of an exact localization sequence in
$K$-Theory. We prefer the explicit description of $\Div(T)$ and
$\de$ given here, as it helps to understand the maps associated with
$\Div(T)$.
\end{remark}

Let $R = Z(T) = K[y] $. So, $q(R) = Z(Q)$. We define $\Div(R)$ 
just as we defined
$\Div(T)$ above. Note that this $\Div(R)$ coincides canonically
with the usual divisor group of fractional ideals of the PID~$R$,
since for $a \in R\setminus \{0\}$, the simple composition factors of
$R/Ra$ are the simple modules $R/P$ as $P$ ranges over the prime
ideal factors of the ideal $Ra$.

\begin{prop} \label{nrdinj} 
For $R= Z(T) = K[Y]$, there is a map $\Nrd \colon \Div(T) \ra \Div(R)$
such that the following diagram commutes:
\begin{equation}\label{nrddiag}
\begin{split}
\xymatrix{ D^* \ar[r] \ar[d]_{\Nrd_D} & Q^* \ar[dd]^{\Nrd_Q}
\ar[r]^{\de_T} &\Div(T) \ar[dd]^{\Nrd} \\
  Z(D)^* \ar[d]_{N_{Z(D)/K}} & & \\
K^*\ar[r] & q(R)^* \ar[r]^{\de_R} & \Div(R) }
\end{split}
\end{equation}  
Moreover, $\Nrd$ is injective.

\end{prop}

\begin{proof}
Let $E= T[x^{-1}]= D[x, x^{-1}, \si]$, which with its grading by
degree in $x$ is a graded division ring with $E_0 = D$ and $q(E) =
Q$. Since $\ind(Q) = \ind(D)\,[Z(D):K]$, by (\ref{NrdD0}),
for $d \in
D^{\ast} = E_0^{\ast}$, $\Nrd_Q(d) = N_{Z(D)/K} (\Nrd_D(d))$. This 
gives
the commutativity of the left rectangle in the diagram.

For the right vertical map in diagram~\eqref{nrddiag}, note that 
there is a canonical map, call
it $N\colon \Div(T) \ra \Div(R)$ given by taking a $T$-module $M$ 
(with ACC and DCC) and viewing it as an $R$-module; that is 
$N(jh_T(M)) =
jh_R(M)$. But, this is not the map $\Nrd\colon\Div(T) \ra \Div(R)$ we
need here! (Consider $N$ a norm map, while our $\Nrd$ is a reduced
norm map.) Note that as $T$ is integral over $R$ and $R$ is
integrally closed, $\Nrd_Q$ maps $T$ into $R$. 
Define a function 
$$
\psi\colon T\setminus\{0\} \to \Div(R) \quad \text{by} \quad 
\psi(f)  \ = \  \delta_R(\Nrd_Q(f))  \ = 
\ jh_R\big(R\big/R\,\Nrd_Q(f)\big). 
$$
Since $\Nrd_Q$ is multiplicative and $\delta _R$ is a 
group homomorphism, we have
\begin{equation}\label{phiformula}
\psi(fg) \ = \ \psi(f) + \psi(g) \quad \text {for all}
\quad f,g\in T\setminus\{0\}.
\end{equation}
We then extend $\psi$ to $Q^*$ by defining $\psi(fr^{-1})
= \psi(f) - \psi(r)$ for all $f\in T\setminus\{0\}$, 
$r \in R\setminus \{0\}$.  Equation \eqref{phiformula} shows
that $\psi$ is well-defined on  $Q^*$ and is a group homomorphism.
Since $\Nrd_Q(D^*) \subseteq K^* \subseteq R^*$ by \eqref{NrdD0}, 
$D^* \subseteq \ker(\psi)$.  Also, $Q' \subseteq \ker(\psi)$
as $\Div(R)$ is abelian.  Thus, by Prop.~\ref{kerde}, $\ker(\delta_T)
\subseteq \ker(\psi)$, so there is an induced homomorphism 
$\Nrd\colon \Div(T) \to \Div(R)$ such that $\Nrd\circ \delta_T
= \psi$ on $Q^*$.  This is the map we need.  Since for every 
$f\in T\setminus\{0\}$, $\Nrd(\delta_T(f)) = \psi(f) = 
\delta_R(\Nrd_Q(f))$, the right rectangle in \eqref {nrddiag} is 
commutative. 

We have a scalar extension map from $R$-modules to $T$-modules given
by $M \ra T\otimes_RM$. This induces a map $\rho \colon\Div(R) \ra
\Div(T)$ given by $\rho(jh_R(M))= jh_T(T\otimes_RM)$. For any 
$r \in R$, we have $T \otimes_R (R/Rr) \cong T/ Tr$. Thus for any $g \in T\setminus
\{0\}$,
\begin{align*}
\rho(\Nrd(\de_T(g)))  \ &= \  \rho(\de_R(\Nrd_Q(g)))  \ =  \ 
\rho(jh_R(R/R\Nrd_Q(g)))   \\
&= \  jh_T(T/T\Nrd_Q(g))  \ = \  \de_T(\Nrd_Q(g))  \ = \  n\,\de_T(g), 
\end{align*} 
using Remark~\ref{deremarks}(iv). 
This shows that $\rho \circ
\Nrd\colon \Div(T) \ra \Div(T)$ is multiplication by $n$, which is an
injection, as $\Div(T)$ is a torsion-free abelian group. Hence
$\Nrd$ must be injective.
\end{proof}

\begin{remark}
 Here is a description of  how the maps 
$\Nrd\colon \Div(T) \to \Div(R)$ and  
$N\colon \Div(T) \to \Div(R)$ and $\rho\colon \Div(R) \to \Div(T)$
are related,  and a formula for  $\Nrd$ on generators of 
$\Div(T)$.
Proofs are omitted.
We have 
\begin{equation}
\rho\circ \Nrd  \ = \  n \id_{\Div(T)};
\end{equation}
and 
\begin{equation}
N \ = \ n\cdot\Nrd.
\end{equation}
Let $S$ be any simple left $T$-module, and $[S]$
the corresponding basic generator of $\Div(T)$. Let $M = \ann_T(S)$,
and let $P = \ann_R(S)$,
which is a maximal ideal of $R$. Let $k$ = matrix of size of $T/M$ = 
$\dim_{\De}(S)$, where
$\De = \End_T(S)$, so $T/M \cong M_k (\De)$. Then,  
\begin{equation}
\Nrd([S])  \ = \  n_S[R/P], \textrm{  \ \ where \ \ } n_S  \ = \  
\textstyle\frac{1}{nk}
\dim_{R/P}(T/M) \ = \  \ind(T/M).
\end{equation}
\end{remark}

We now consider an arbitrary graded division ring $E$.  As usual,
we assume throughout that $\Gamma_E$  is a torsion-free abelian 
group and $[E:Z(E)] <\infty$.

\begin{lemma} \label{skeskqinj}
Let $E$ be a graded division ring, and let $Q=q(E)$. Then, the 
canonical map 
$\SK(E) \ra \SK(Q)$ is injective.
\end{lemma}

\begin{proof}
Recall from Prop.~\ref{normfacts}(i) that $\Nrd_E(a)=\Nrd_Q(a)$ for all 
$a\in E$, so the inclusion ${E^* \hookrightarrow Q^*}$ yields 
 a map  
$\SK(E) = E^{(1)}/ E' \ra Q^{(1)} / Q'=\SK(Q) $. 
Also recall the homomorphism  $\la\colon Q^* \rightarrow E^*$ of  
(\ref{lamfun}), which 
maps $Q'$ to $E'$. Since the composition
$E^{\ast} \hookrightarrow Q^{\ast} \stackrel{\la}{\rightarrow} E^*$
 is the identity map, for any $a\in E^{(1)}\cap Q'$, we have 
$a = \lambda(a) \in E'$.  Thus, the map
$\SK(E) \ra \SK(Q)$ is injective.
\end{proof}

\begin{proposition} \label{iso}
Let $E$ be a graded division ring, and let $Q=q(E)$. Then,
$$
 Q^{(1)}  \ = \  (Q^{(1)} \cap E_0)Q'.
$$
\end{proposition}

Once this proposition is proved, it will quickly yield the main 
theorem of this section:

\begin{thm}\label{isothm}
Let $E$ be a graded division ring. 
Then, $\SK(E) \cong \SK(q(E))$.
\end{thm}

\begin{proof} Set $Q=q(E)$. Since the reduced norm respects 
scalar extensions, $Q^{(1)} \cap E_0 \subseteq E^{(1)}$. The image
of the 
 map $\xi\colon\SK(E) \ra \SK(Q)$ is $E^{(1)}Q'/Q'$, which thus 
contains $(Q^{(1)} \cap E_0)Q' / Q' = Q^{(1)}/ Q' = \SK(Q)$ 
(using Prop.~\ref{iso}). Thus $\xi$ is surjective, as well as 
being injective  by Lemma ~\ref{skeskqinj},  proving the theorem.
\end{proof}

\begin{proof}[Proof of Prop.~\ref{iso}]
 We first treat the  case where 
 $\Ga_E$ is finitely generated.

{ Case I.}
{\it Suppose $\Ga_E = \mathbb{Z}^{n}$ for some $n \in \mathbb N$.}

Let $F= Z(E)$,
a graded field, and let $\ep_i = (0, \ldots, 0 ,1, 0, \ldots, 0)$
($1$ in the i-th position), so $\Ga_E = \mathbb{Z} \ep_1 \oplus
\ldots \oplus \mathbb{Z}\ep_n$. For $1 \leq i \leq n$, let $\Delta_i
= \mathbb{Z} \ep_1 \oplus \ldots \oplus \mathbb{Z}\ep_i \subseteq
\Ga_E;$ and let $S_i = E_{\Delta_i} =  \bigoplus_{\ga \in \Delta_i}
E_{\ga}$, which is a graded sub-division ring of $E$. Let $Q_i =
Q(S_i)$, the quotient division ring of $S_i$; so $Q_n =Q$ as $S_n =
E$. Set $R_0 = Q_0 = E_0$. Note that $[S_i : (S_i \cap F)] < \infty$,
so $Q_i$~is obtainable from $S_i$  by inverting the nonzero elements
of $S_i \cap F$. This makes it clear that $Q_i \subseteq Q_{i+1}$,
for each $i$.

For each $j$, $1 \leq j \leq n$, choose and fix a nonzero element
 $x_j \in  E_{\ep_j}$. Let ${\ph_j = \intt(x_j) \in \Aut(E)}$ (i.e.,
$\ph_j$ is conjugation by $x_j$). Since $\ph_j$ is a degree-preserving automorphism of $E$, $\ph_j$ maps each $S_i$ to itself.
Hence, $\ph_j$ extends uniquely to an automorphism to $Q_i$, also
denoted $\ph_j$. Since each $\Ga_E/ \Ga_F$ is a torsion abelian group,
there is $\ell_j \in \mathbb  N$ such that $\ell_j \ep_j \in \Ga_F$. Then, if
we choose any nonzero $z_j \in F_{\ell_j \ep_j}$, we have
$x_j^{\ell_j} \in E_{\ell_j \ep_j} = E_0 z_j$. So,
$x_j^{\ell_j} = c_j z_j$ for some $c_j \in E_0^{\ast}$, and $z_j
\in F = Z(E)$. Then $\ph_j^{\ell_j} = 
\intt({x_j}^{l_j})=\intt(c_j z_j) = \intt(c_j)$.
Thus, $\ph_j^{\ell_j}|_{S_i}$ is an inner automorphism of $S_i$ for
each $i$, as $c_j \in E_0^{\ast} \subseteq S_i$.

Now, fix $i$ with $1 \leq i \leq n$. We will prove:
\begin{equation} \label{eqn}
Q_i^{\ast} \cap
Q^{(1)}  \ \subseteq \  (Q_{i-1}^{\ast} \cap Q^{(1)}) [Q_i^{\ast},
Q^{\ast}].
\end{equation}

We have $S_i = S_{i-1}[x_i, x_i^{-1}] \cong S_{i-1} [x_i,
x_i^{-1}, \ph_i]$ (twisted Laurent polynomial ring). Likewise,
within $Q_i$ we have $Q_{i-1}[x_i] \cong Q_{i-1}[x_i, \ph_i]$
(twisted polynomial ring), with $\ph_i^{\ell_i}$ an inner
automorphism of $Q_{i-1}$. In order to invoke Prop.~\ref{nrdinj},
let 
$$ 
T  \, = \  Q_{i-1} [x_i]  \ \cong  \ Q_{i-1}[x_i, \ph_i]\quad
\text{and let} \quad R= Z(T).
$$
Since $S_{i-1}[x_i] \subseteq T \subseteq Q_i =
q(S_{i-1}[x_i])$, we have $q(T) = Q_i$.
 Let $\mathcal{G} \subseteq \Aut(Q_i)$ be the subgroup of 
automorphisms of
$Q_i$ generated by $\ph_{i+1}, \ldots, \ph_n$, and let $G= \maj /
(\maj \cap \Inn(Q_i))$, where $\Inn(Q_i)$ is the group of inner
automorphisms of $Q_i$. 
Since  Skolem-Noether shows that $\Inn(Q_i)$ is the kernel of the 
restriction map $\Aut(Q_i) \to \Aut(Z(Q_i))$, this $G$ maps injectively 
into $\Aut(Z(Q_i))$.  For $\sigma \in G$, we write 
$\sigma|_{Z(Q_i)}$ for the automorphism of $Z(Q_i)$ determined by $\sigma$.
Note that $G$~is a finite abelian group,
since the images of the $\ph_i$ have finite order in $G$ and commute
pairwise. (For, we have $x_j x_k = c_{jk} x_k x_j$ for
some $c_{jk} \in E_0^{\ast}$. Hence $\ph_j \ph_k = \intt(c_{jk}) \ph_k
\ph_j$ and $\intt(c_{jk}) \in \Inn(Q_i)$, as $c_{jk} \in E_0^{\ast}
\subseteq Q_i^{\ast}$). 
Every element of $\maj$ is an automorphism of 
$S_{i-1}[x_i]$ preseerving degree in $x_i$, so an automorphism of
$T$, since this is true of each $\ph_j$. Therefore we have a group
action of $\maj$ on $T$ by ring automorphisms,
 and an induced action of $\maj$ on~
$\Div(T)$. 
 Note that as any $\psi \in \maj$ permutes the maximal left
ideals of $T$, the action of $\psi$ on $\Div(T)$ arises from an
action on the base of $\Div(T)$ consisting of isomorphism classes of
simple $T$-modules. That is, $\Div(T)$ is a permutation
$\maj$-module.  $\maj$ also acts on $R= Z(T)$ 
by ring automorphisms, and on
$\Div(R)$, and  all the maps in the commutative
diagram below (see Prop.
~\ref{nrdinj}) are $\maj$-module homomorphisms.
\begin{equation}\label{diagndr}
\begin{split}
\xymatrix{ Q_i^{\ast} \ar[r]^{\de_T} \ar[d]_{\Nrd_{Q_i}} & \Div(T)
\ar[d]^{\Nrd} \\
Z(Q_i)^{\ast} \ar[r]^{\de_R} & \Div(R) }
\end{split}
\end{equation}
Since inner automorphisms of $Q_i$ act trivially on $\Div(T)$ (see
Remark~\ref{deremarks}(iii)), and on $Z(Q_i)$ and  $\Div(R)$, these 
$\maj$-modules are actually $G$-modules. Let
$$
\mfn \ = \  \Nrd(\Div(T))  \ \subseteq  \ \Div(R).
$$
Because $\Nrd\colon\Div(T) \ra \Div(R)$ is injective (see Prop.
~\ref{nrdinj}), $\mfn$ is a $G$-module isomorphic to $\Div(T)$, so
$\mfn$ is a permutation $G$-module. In $\mfn$ we have two
distinguished $\maj$-submodules,
\begin{align*}
\mfn_0  \ &=  \ \ker (N_G), \text{where $N_G\colon \mfn \ra \mfn$ 
is the
norm, given by $N_G(b) = \textstyle\sum_{\si \in G} \si(b)$; and} \\
I_G(\mfn)  \ &=  \ \big\langle \{\beta - \si(\beta) \,|\  
\beta \in \mfn, \, 
\si \in G
\} \big\rangle \subseteq \mfn_0.
\end{align*}

By definition, ${\widehat{H}}^{-1}(G, \mfn) = \mfn_0 / I_G(\mfn)$. 
But, because $\mfn$ is a permutation $G$-module, 
${\widehat{H}}^{-1}(G, \mfn)
= 0$. (This is well known, and is an easy calculation, as 
$\mfn$ is a direct sum of $G$-modules of the form 
$\mathbb Z[G/H]$ for subgroups $H$ of $G$.) That is,
$\mfn_0 = I_G(\mfn)$.

Take any generator $\beta - \si(\beta)$ of $I_G(\mfn)$, where $\si \in G$ 
and $\beta
\in \mfn$, say $\beta= \Nrd(\eta)$, where $\eta \in \Div(T)$. Take any
$b \in Q_i^{\ast}$ with $\de_T(b)= \eta$, and 
choose $u \in E^{\ast}$
which is some product of the~$\varphi_j$ $(i+1 \leq j \leq n)$, such
that $\intt(u)|_{Z(Q_i)} = \si|_{Z(Q_i)}$. 
Then, $\de_R(\Nrd_{Q_i}(b)) =
\Nrd(\de_T(b)) = \beta$ 
(see \eqref{diagndr}). Also, because $\intt(u)|_{Q_i}$ is an
automorphism of $Q_i$, we have $\Nrd_{Q_i}(ub^{-1}u^{-1}) = u\,
\Nrd_{Q_i}(b^{-1}) u^{-1}$. Thus, $bub^{-1}u^{-1} \in [Q_i^*,Q^*]
\cap Q_i$ and 
\begin{align*}
\Nrd_{Q_i}(bub^{-1}u^{-1})  \ &= \  \Nrd_{Q_i}(b)\,
\Nrd_{Q_i}(ub^{-1}u^{-1}) \\
&=  \ \Nrd_{Q_i}(b)\,u\, \Nrd_{Q_i}(b^{-1}) u^{-1} 
 \ = \ \Nrd_{Q_i}(b) \big/ \si(\Nrd_{Q_i}(b)). \nono
\end{align*}
Hence, in $\Div(R)$,
$$
 \de_R\big(\Nrd_{Q_i}(bub^{-1}u^{-1})\big)  \ =  \ 
\de_R \big(\Nrd_{Q_i}(b)/ \si \Nrd_{Q_i}(b)\big)  \ = \  \beta- \si(\beta).
$$ 
Since such $\beta -\sigma(\beta)$ generate $I_G(\mfn)$, it follows
that for any $\ga \in I_G(\mfn)$, there is $c \in [Q_i^{\ast},
Q^{\ast}] \cap Q_i$, with $\ga =\de_R(\Nrd_{Q_i}(c)) 
=\Nrd(\de_T(c))$  (see\eqref{diagndr}).

To prove (\ref{eqn}), we need a formula for $\Nrd_Q$ for an element of
$Q_i$. For this, note that ${E= S_i [ x_{i+1},  x_{i+1}^{-1},
\ldots , x_n, x_n^{-1}]}$ which can be considered a graded ring
over $S_i$. 
Now, let \break $C = Q_i[x_{i+1},  x_{i+1}^{-1}, \ldots ,
x_n, x_n^{-1}] \subseteq Q$. This $C$  is a graded division
ring with $C_0 = Q_i$ and $\Ga_C = \mathbb{Z} \ep_{i+1} \oplus \ldots \oplus
\mathbb{Z} \ep_n$. Since $E \subseteq C \subseteq Q
= q(E)$, we have $q(C) = Q$. For the graded field $Z(C)$ we have
$Z(C)_0$ consists of those elements of $Z(C_0) = Z(Q_i)$
centralized by $x_{i+1}, \ldots , x_n$, i.e., 
$Z(C)_0$ is the fixed field $Z(Q_i)^{\maj} = Z(Q_i)^G$. 
Since, as noted earlier $G$ injects into $\Aut(Z(Q_i)$, we have
$G \cong \Gal(Z(Q_i)/Z(C)_0)$.
Thus, for any $q \in Q_i = C_0$, 
by Prop.~\ref{normfacts}(i)~and~(iv), 
\begin{align*}
\Nrd_Q(q)  \ &=  \ \Nrd_{q(C)}(q)  \ = \  \Nrd_C(q)  
 \ = \  N_{Z(C_0)/ Z(C_0)^G}(\Nrd_{C_0}(q))^{m} \\
&= \ N_{Z(Q_i)/ Z(Q_i)^G}(\Nrd_{Q_i}(q))^{m}, 
\end{align*}
where $m = \ind(Q)/ \ind(Q_i)[Z(Q_i):Z(Q_i)^G].$ 

To verify (\ref{eqn}), take any $a \in Q_i^{\ast} \cap Q^{(1)}$. 
Thus,
$$ 
1  \, =  \ \Nrd_Q(a)  \ = \  N_{Z(Q_i) / Z(Q_i)^G}(\Nrd_{Q_i}(a))^m. 
$$
Hence, for $\al = \de_T(a) \in \Div(T)$, using the identification of 
$G$ with $\Gal(Z(Q_i)/Z(C)_0)$ and commutative 
diagram~\eqref{diagndr}, 
\begin{align*}
0  \ &=  \ \de_R(\Nrd_Q(a)) \ = \  
\de_R \big(N_{Z(Q_i)/Z(Q_i)^G}(\Nrd_{Q_i}(a))^m\big)  \ =  \ 
\textstyle\sum\limits_{\sigma \in G} 
\sigma\big(\de_R(\Nrd_{Q_i}(a)^m)\big)\\
 \ &= \  N_G\big(\de_R(\Nrd_{Q_i}(a))^m\big) \ = 
\ m \, N_G(\Nrd(\de_T(a)))
 \ = \ m \,N_G(\Nrd(\al)).
\end{align*}  
Since $\Div(R)$ is  torsion-free, we have
$N_G(\Nrd(\al)) = 0$, i.e., $\Nrd(\al) \in \ker(N_G) = \mfn_0 = I_G(\mfn)$. 
Therefore, as we saw above, there is $c \in
[Q_i^{\ast}, Q^{\ast}] \cap Q_i^*$ with 
$\Nrd(\al) = \Nrd(\de_T(c))$. Let $a' = a/c \in Q_i^{\ast}$. 
Then,
$$ 
\Nrd(\de_T(a'))  \ = \ \Nrd(\de_T(a)) -\Nrd(\de_T(c)) \ 
   = \ \Nrd(\al)- \Nrd(\al)  \, = \,  0.
$$
 Because $\Nrd \colon \Div(T) \ra \Div(R)$ is
injective (see Prop.~\ref{nrdinj}), it follows that $\de_T(a') = 0$
in~$\Div(T)$. Therefore, as $T = Q_{i-1}[x_, \varphi_i]$ and 
$q(T) = Q_i$,
by Prop. ~\ref{kerde} there is $a'' \in
 Q_{i-1}$ with $a'' \equiv a' \ (\text{mod }Q'_i)$. So, 
${a'' \equiv a}$
 $\ (\text{mod }[Q_i^{\ast}, Q^{\ast}])$, and hence $\Nrd_Q(a'') = \Nrd_Q(a)
=1$, i.e., $a'' \in Q_{i-1}^{\ast} \cap Q^{(1)}$. Thus,
${a \in
(Q_{i-1}^{\ast} \cap Q^{(1)})[Q_{i}^{\ast}, Q^{\ast}]}$, proving
(\ref{eqn}).

The inclusion (\ref{eqn}) shows that for any $i$, $1 \leq i \leq n$
and any $a \in Q^{(1)} \cap Q_i$ there is $b \in Q^{(1)} \cap Q_{i
-1}$ with $b \equiv a \ (\text{mod }Q')$. Hence, by downward induction on
$i$, for any $q \in Q^{(1)} = Q^{(1)} \cap Q_n$ there is $d \in Q_0
\cap Q^{(1)} = E_0 \cap Q^{(1)}$ with $d \equiv q 
\ \text{mod }Q')$. So,
$Q^{(1)} \subseteq (Q^{(1)} \cap E_0)Q'$. The reverse inclusion is
clear, completing the proof of Case I.

{ Case II.} {\it Suppose $\Ga_E$ is not a finitely generated
abelian group.}

 The basic point is that $E$ is a direct limit of
sub-graded division algebras with finitely generated grade group, so
we can reduce to Case I. But we need to be careful about the choice
of the sub-division algebras to assure that they have the same index
as $E$, so that the reduced norms are compatible.

Let $F= Z(E)$. Since $| \Ga_E / \Ga_F| < \infty$, there is a finite
subset, say $\{ \ga_1, \ldots, \ga_k \}$ of $\Ga_E$ whose images in $
\Ga_E/ \Ga_F$ generate this group. Let $\De_0$ be any finitely
generated subgroup of $\Ga_E$, and let $\De$ be the subgroup of
$\Ga_E$ generated by $\De_0$ and $\ga_1 , \ldots, \ga_k$. Then,
$\De$ is also a finitely generated subgroup of $\Ga_E$, but with the
added property that $\De + \Ga_F = \Ga_E$. Let
$$ 
E_{\De}  \ = \ \textstyle \bigoplus\limits_{\de \in \De} E_{\de},
$$ 
which is a graded
sub-division ring of $E$, with $E_{\De, 0} = E_0$ and $\Ga_{E_{\De}}
= \De$. Since $\De + \Ga_F = \Ga_E$, we have $E_{\De}F= E$. (For,
take any $\ga \in \Ga_E$ and write $\ga= \de +\eta$ with $\de \in
\De$ and $\eta \in \Ga_F$, and any nonzero $d \in
E_{\De, \de}$ and $c \in F_{\eta}$. Then, $E_{\ga} = dc E_0
\subseteq E_{\De}F$.) Because $E_{\De}F = E$, we have $Z(E_{\De}) =
F \cap E_{\De} = F_{\De \cap \Ga_F}$. Note that
\begin{align*}
[E_{\De}:Z(E_{\De})]  \ &= \  [E_{\De, 0}: F_{{\De \cap\Ga_F},
0}] \, |\Ga_{\De}: (\Ga_{\De} \cap \Ga_{F})| 
 \ = \  [E_0:F_0] \, |( \Ga_{\De} + \Ga_F ): \Ga_F|  \\
&= \ [E_0 : F_0]  \, |\Ga_E : \Ga_F|  \ = \  [E:F] .
\end{align*}
The graded homomorphism $E_{\De} \otimes_{Z(E_{\De})} F \ra E$ is
onto as $E_{\De} F = E$, and is then also injective by dimension
count (or by the graded simplicity of $E_{\De} \otimes_{Z(E_{\De})}
F$). Thus, $E_{\De} \otimes_{Z(E_{\De})} F \cong E$. It follows that
$q(E_{\De}) \otimes_{q(Z(E_{\De}))} q(F) \cong q(E)$. Specifically,
\begin{align*}
q(E_{\De}) \otimes_{q(Z(E_{\De}))} q(F)  \ & \cong  \  (E_{\De}
\otimes_{Z(E_{\De})} q(Z(E_{\De})))\otimes_{q(Z(E_{\De}))} q(F)
 \ \cong \   E_{\De}\otimes_{Z(E_{\De})} q(F)  \\
& \cong  \  (E_{\De} \otimes_{Z(E_{\De})} F) \otimes_F q(F) 
 \  \cong  \  E \otimes_F q(F)  \ \cong \  q(E). 
\end{align*}
Therefore, for any $a \in q(E_{\De})$, $\Nrd_{q(E_{\De})}(a) =
\Nrd_{q(E)}(a)$.

Now, if we take any $a \in Q^{(1)}$ where $Q = q(E)$, there is a
subgroup $\De \subseteq \Ga_E$ with $\De$ finitely generated and
$\De + \Ga_F = \Ga_E$ and $a \in E_{\De}$. Since
$\Nrd_{q(E_{\De})}(a) = \Nrd_Q(a)=1$, we have, by Case~I applied to
$E_{\De}$, $a \in \big(q(E_{\De})^{(1)} \cap E_0\big)q(E_{\De})' \subseteq
(Q^{(1)} \cap E_0)Q'$, completing the proof for Case II.
  \end{proof}

\begin{remark}
(i)  Prop.~\ref{iso} for those $E$ with $\Ga_E \cong
\mathbb{Z}$ was proved in \cite{py}, and our proof of this is
essentially the same as theirs, expressed in a somewhat different
language. Platonov and {Yanchevski\u\i} also in effect assert 
Prop. ~\ref{iso} for $E$ with $\Ga_E$ finitely generated, expressed
as a result for iterated quotient division rings of twisted
polynomial rings. (See \cite{py}, Lemma~8.) 
By way of proof of
\cite{py}, Lemma~8, the authors say nothing more than that it follows by
induction from the rank 1 case. It is not clear whether the proof
given here coincides with their unstated proof, since the transition
from rank 1 to finite rank is not transparent.

(ii) So far the functor $\CK$ has manifested   properties similar to
$\SK$.  However, the similarity does not hold here, since  
the functor $\CK$ is not (homotopy) stable. In fact, for a 
division algebra~$D$ over its center $F$ of index $n$, one has 
 the following split exact sequence,
$$
1 \rightarrow \CK(D) \rightarrow \CK(D(x)) \rightarrow 
\textstyle\bigoplus\limits_p \mathbb Z/ (n/n_p) \mathbb Z \rightarrow 1
$$
where $p$ runs over irreducible monic polynomials of $F[x]$ and 
$n_p$ is the index of central simple algebra  
$D\otimes_F \big(F[x]/(p)\big)$ 
(see Th.~2.10 in \cite{sk12001}).  This is provable by 
mapping the exact sequence \eqref{locseq} with $T = F[x]$ to 
the sequence for $T = D[x]$ and taking cokernels.
\end{remark}

\begin{example}
Let $E$ be a semiramified graded division ring with $\Gamma_E\cong
\zz^n$, and let $T = Z(E)$.  Since $\Gamma_E/\Gamma_T$ is a torsion
group, there are a base $\{\gamma_1, \ldots, \gamma_n\}$ of the free
abelian group $\Gamma_E$ and some $r_1, \ldots, r_n\in \nn$ such that 
$\{r_1\gamma_1, \ldots, r_n\gamma_n\}$ is a base of $\Gamma_T$.  Choose 
any nonzero $z_i \in E_{\gamma_i}$ and $x_i \in T_{r_i\gamma_i}$, 
$1\le i\le n$.  Let $F= T_0$ and $M = E_0$, and 
let $G = \Gal(M/F)$.  Because $E$ is semiramified,
$M$ is Galois over $F$ with ${[M:F] = |\Gamma_E:\Gamma_T| = \ind(E)
= r_1\ldots r_n}$, and ${G\cong \Gamma_E/\Gamma_T}$.  Since $z_i^{r_i}
\in E_{r_i\gamma_i} = E_0x_i$, there is $b_i\in M$ with $z_i^{r_i}
= b_ix_i$.  Let ${u_{ij} = z_iz_jz_i^{-1}z_j^{-1}\in M}$. Let
$\sigma_i\in G$ be the automorphism of $M$ determined by conjugation
by $z_i$.  From the isomorphism ${G\cong \Gamma_E/\Gamma_T}$, 
each $\sigma_i$ has order $r_i$ in $G$ and $G\cong 
\langle \sigma_1\rangle \times \ldots \times \langle \sigma_n\rangle$. 
Clearly, ${T = F[x_1,x_1^{-1}, \ldots x_n, x_n^{-1}]}$, an iterated
Laurent polynomial ring, and $E = M[z_1, z_1^{-1}, \ldots, z_n,
z_n^{-1}]$, an iterated twisted Laurent polynomial ring whose multiplication 
is completely determined by the $b_i\in M$, the $u_{ij}\in M$, and the action of the 
$\sigma_i$ on $M$.  

Let $D = q(E)$, which is a division ring with center $q(T) = 
F(x_1, \ldots , x_n)$, a rational function field over $F$.
Then, $D$ is the {\it generic abelian crossed product} determined by 
$M/F$, the base $\{\sigma_1,\ldots, \sigma_n\}$ of $G$, the $b_i$ and 
the $u_{ij}$, as defined in \cite{as}.
  As was pointed out in \cite{bm}, all generic abelian crossed products 
arise this way as rings of quotients of semiramified graded division
algebras.  Generic abelian crossed products were used in \cite{as} to 
give the first examples of noncyclic $p$-algebras, and in \cite{s1}
to prove the existence of noncrossed product $p$-algebras.
It is known by \cite{tignol}, Prop.~2.1 that $D$ is determined up to
$F$-isomorphism by $M$ and the $u_{ij}$.    
By Cor.~\ref{skunramthm}(iii) and Th.~\ref{isothm}, there is 
an exact sequence
\begin{equation}\label{semiramseq}
G\textstyle\wedge G  \ \to  \  \widehat H^{-1}(G,M^*) 
 \ \to \  \SK(D) \  \to \  1,
\end{equation}
where the left map is determined by sending
$\sigma_i\wedge \sigma_j$ to $u_{ij}$ mod $I_G(M^*)$.
An important  condition introduced by Amitsur and Saltman in \cite{as} was 
{\it nondegeneracy} of   $\{u_{ij}\}$.  This condition was 
essential  for the noncyclicity results in \cite{as}, and is 
also key to the results on noncyclicity and indecomposability of generic abelian 
crossed products in recent work of 
McKinnie in \cite {mc1}, \cite {mc2} and Mounirh \cite{mou2}.  
The original definition of nondegeneracy in \cite{as} was 
somewhat mysterious.  A cogent characterization was given
recently in \cite{mc3}, Lemma~5.1:  A family $\{u_{ij}\}$ 
in~$M^*$ (meeting the 
conditions to appear in a generic abelian crossed product) is 
nondegenerate iff for every rank~$2$ subgroup $H$ of $G$, the map 
$H\wedge H \to \widehat H^{-1}(H, M^*)$ appearing in the complex
\eqref{semiramseq} for the generic abelian crossed product
$C_D(M^H)$ is nonzero.  In the first nontrivial case, where 
$G\cong \zz_p\times \zz_p$ with $p$ a prime number, we have
$\{u_{ij}\}$ is nondegenerate iff the map $G\wedge G\to 
\widehat H^{-1}(G,M^*)$ is nonzero, iff the epimorphism
$\widehat H^{-1}(G,M^*) \to \SK(D)$ is not injective.  
Thus, the nondegeneracy is encoded in $\SK(D)$, and it occurs
just when $\SK(D)$ is not \lq\lq as large as possible."

\end{example}

\appendix
\section{The Wedderburn factorization theorem} \label{weddapp}

In a division ring, additive and multiplicative commutators play 
important roles and there are extensive results in the literature 
known as commutativity theorems. The main theme in these results 
is that, additive and multiplicative commutators are ``dense'' in 
a division ring.  For example, if an element commutes with all 
additive commutators, then it is already a central element. It 
seems that this 
trend continues  for the additive commutators for a graded 
division ring. However the multiplicative commutators are 
too ``isolated'' to determine the structure of a graded division 
ring.

Let $E$ be  a graded division ring with graded center $T$. A
\emph{homogeneous additive commutator} of $E$ is an element of the
form $ab-ba$ where $a,b \in E^{h}$. We will use the notation 
$\lbr a,b\rbr_{\textrm{ad}} 
= ab-ba$ for $a,b \in E^{h}$ and let $\lbr H, K\rbr_{\textrm{ad}}$ 
be the additive 
group generated by 
$\{ dk-kd : d \in
H^h, k \in K^h\}$ where $H$ and $K$ are graded subrings of $E$.
Parallel to  the theory of division rings, one can show that if all
the homogenous additive commutators of graded division ring $E$ are
central, then $E$ is a graded field. To observe this, one can carry
over the non-graded proof, {\it mutatis-mutandis}, to the graded
setting, see, e.g., \cite{lam}, Prop.~13.4. Alternatively, let $y \in E^h$ be an element which commutes
with homogeneous additive commutators of $E$. Then  $y$ commutes
with all (non-homogeneous) commutators of $E$. Consider $\lbr x_1 ,
x_2\rbr_{\textrm{ad}}$ where $x_1 , x_2 \in q(E)$. Since $q(E)=E\otimes_T q(T)$, 
it follows
that $y\lbr x_1 , x_2 \rbr_{\textrm{ad}} = \lbr x_1 , x_2 \rbr_{\textrm{ad}} y$. So $y$ commutes
 with all
commutators of $q(E)$, a division ring, thus $y \in q(T)$. But $
E^h \cap
q(T) \subseteq T^h$, proving that $y \in T^h$. Thus, $E$~is
commutative. Again parallel to the theory of division rings, one can
prove that if $K \subseteq E$ are graded division rings, with 
$\lbr E,K\rbr_{\textrm{ad}} 
\subseteq K$ and $\chr(K) \neq 2$, then $K \subseteq Z(E)$. However,
for this one it seems there is no shortcut, and one needs to 
carry out a proof similar to the one for ungraded
 division rings, as in  (\cite{lam}, Prop.~3.7).

The  paragraph above shows some similar behavior between the Lie
algebra structure of division rings and that of graded division rings.
However, this analogy often fails for the multiplicative structure 
of graded division algebras. For example, the Cartan-Brauer-Hua
theorem (the multiplicative analogue of the  statement above that  if $K \subseteq E$ are graded division rings, with 
$\lbr E,K\rbr_{\textrm{ad}} 
\subseteq K$ and $\chr(K) \neq 2$, then $K \subseteq Z(E)$) is not
valid in the graded setting. Also, the  multiplicative group $E^*$ 
of a
totally ramified graded division algebra $E$ is nilpotent
(since $E' \subseteq E_0^* =T_0^* \subseteq Z(E^*)$),
while the  multiplicative group of a noncommutative division ring 
is not even
solvable, cf.~\cite{stuth}.
Furthermore,  a totally ramified graded division algebra
 $E^*$ is radical over its
center $T$ (since $E^{*\exp(\Gamma_E/\Gamma_T)} \subseteq T^*$), but 
this is not the case for any non-commutative division
ring (\cite{lam}, Th.~15.15). Nonetheless, one significant
  theorem involving conjugates that can be extended to
the graded setting is the Wedderburn factorization theorem. (This is
used in proving Th.~\ref{normalthm}.)


\begin{theorem}[Wedderburn Factorization Theorem]\label{domainW}
Let $E$ be a graded division ring with center $T$ (with 
$\Gamma_E$ torsion-free abelian).
Let  $a$ be 
a homogenous element   of $E$ which is algebraic over $T$ with  
minimal 
polynomial $h_a \in T[x]$.  
Then, $h_a$ splits completely in $E$. Furthermore, there exist 
$n$ conjugates $a_{1}, \ldots , a_{n} $  of $a$ such
that  $h_a = (x-a_{n})(x-a_{n-1}) \ldots (x-a_{1})$ in $E[x]$.
\end{theorem}

\begin{proof} The proof is similar to Wedderburn's  original proof 
for a division ring (\cite{wedd}, 
see also~\cite{lam} for a nice account of the proof). 
We sketch the proof for the convenience of the reader.
For $f = \sum c_i x^i\in E[x]$ and $a\in E$, our convention is 
that $f(a)$ means $\sum c_ia^i$. Since $\Gamma_E$ is torsion-free,
we have $E^* = E^h\setminus\{0\}$.
\begin{description}
\item[I] Let $f \in E[x]$ with factorization  $f=gk$ in  $E[x]$. 
If $a\in E$ such that
$k(a) \in T\cdot E^*$, then $f(a) =
g(a')k(a)$, for some conjugate  $a'$ of $a$.
(Here $E$ could be  any ring with $T\subseteq Z(E)$.)
\end{description}

\begin{proof}
Let $g = \sum b_i x^i$. Then, $f = \sum b_i k x^i $, so
$f(a) = \sum b_i k(a) a^i$. But, $k(a)=te$, where $t \in T$ and 
$e\in E^*$. Thus,  $f(a)=\sum b_i te a^i = \sum
b_i e a^i e^{-1} te = \sum b_i (e a e^{-1})^{i} te = g(ea
e^{-1})k(a)$.
\end{proof}

\begin{description}
\item[II]
Let $f \in E[x]$ be a non-zero polynomial (here $E$ could be, 
in fact,  any ring). Then $r \in E$ is a root of $f$ if and
only if $x-r$ is a right divisor of $f$ in $E[x]$.
(Here $E$ could be  any ring.)
\end{description}

\begin{proof}
We have $x^i-r^i = (x^{i-1} + x^{i-2}r + \ldots + r^{i-1})(x-r)$
for any $i\ge1$.  Hence, 
\begin{equation}\label{x-r}
f - f(r) = g\cdot (x-r) 
\end{equation} for some 
$g\in E[x]$.  So, if $f(r)= 0$, then $f = g\cdot(x-r)$.
Conversely, if $x-r$ is a right divisor of $f$, then
equation~\eqref{x-r} shows that $x-r$ is a right divisor of the 
constant $f(r)$.  Since $x-r$ is monic, this implies that $f(r) = 0$. 
\end{proof}

\begin{description}
\item[III]
If a non-zero monic polynomial $f \in E[x]$
vanishes identically on the conjugacy class $A$ of~$a$ (i.e., 
$h(b)=0$ for all  $b \in A$),then $\deg(f) \geq \deg(h_a)$.
\end{description}

\begin{proof}
Consider $f = x^m + d_1
x^{m-1} + \ldots + d_m \in E[x]$  such that $f(A)=0$ and $m <
\deg(h_a)$ with $m$ as small as possible. 
Suppose $a\in E_\gamma$, so $A\subseteq E_\gamma$, as
the units of $E$ are all homogeneous. Since the 
$E_{m\gamma}$-component
of $f(b)$ is $0$ for each $b\in A$, we may assume that each
$d_i \in E_{i\gamma}$. 
Because $f \notin T[x]$,
some $d_i \notin T$. 
Choose $j$ minimal with $d_j\notin T$, and some  $e \in E^*$
such that $e d_j \neq d_j e$. For any $c \in E$, write $c' :=
ece^{-1}$. Thus $d'_j \neq d_j$ but $d_\ell' = d_\ell$
for $\ell<j$. Let 
$f' = x^m +d_1'x^{m-1} +\ldots +d_m' \in E[x]$.
 Now, for all $b \in A$, we have  $f'(b') =[f(b)]'= 0' = 0$. Since, 
$eAe^{-1} = A$, this shows that $f'(A) = 0$.  Let $g = f-f'$, 
which has degree $j < m$ with leading coefficient $d_j-d_j'$.  
Then, $g(A) = 0$.  But, $d_j-d_j' \in E_{j\gamma}\setminus \{0\}
\subseteq E^*$.  Thus, $(d_j-d_j')^{-1}g$ is monic of degree $j <m$
in $E[x]$, 
and it vanishes on $A$.  This contradicts the choice of $f$; hence,
$m\ge \deg(h_a)$. 
\end{proof}

We now prove the theorem. Since $h_a(a)=0$, by (II),  
$h_a \in E[x]\cdot(x-a)$. Take a factorization
$$
h_a = g \cdot (x-a_r) \ldots (x-a_1)
$$
where $g \in E[x]$, $a_1, \ldots, a_r \in A$ and $r$ is as large
as possible. Let $k = (x-a_r) \ldots (x-a_1)\in E[x]$. 
We claim that
$k(A) = 0$, where $A$ is the conjugacy class of $a$. For,
suppose there exists $b \in A$ such that $k(b) \neq
0$.  Since $k(b)$ is homogenous, we have $k(b) \in E^*$. 
But, $h_a = gk$, and $h_a(b) = 0$, as $b\in A$; hence,   (I)
implies that $g(b')=0$ for some  conjugate $b'$ of $b$. We can then
write $g = g_1 \cdot (x-b')$, by (II). So $h_a$ has a right factor
$(x-b')k = (x-b')(x-a_r) \ldots (x-a_1)$, contradicting our choice 
of $r$.
Thus $k(A)=0$, and using (III), we have $r\ge\deg(h_a)$, which says 
that
$h_a = (x-a_r) \ldots (x-a_1)$.
\end{proof}

\begin{remark}[Dickson Theorem]
One can also see that, with the same assumptions as in
Th.~\ref{domainW}, if $a, b \in E$ have the same minimal
polynomial $h\in T[x]$, then $a$ and $b$ are conjugates. For,
$h=(x-b)k$ where $k \in T[b][x]$. But then by (III),
there exists a conjugate of $a$, say $a'$,  such that $k(a') \not
=0$. Since $h(a')=0$,  by (I) some  conjugate of $a'$ is a root
of $x-b$. (This is also deducible using the graded version of the 
Skolem-Noether theorem, see \cite{hwcor}, Prop.~1.6.)
\end{remark}

\section{The Congruence theorem for 
 tame division algebras}\label{apencongru}

For a valued division algebra $D$, the  congruence
theorem provides a bridge for relating the reduced Whitehead group of
$D$ to the reduced Whitehead group of its residue division
algebra. This was used by Platonov \cite{platonov}  to produce
non-trivial examples of $\SK(D)$, by carefully choosing $D$ with a
suitable residue division algebra. Keeping the notations of
Section ~\ref{prel}, Platonov's congruence theorem states that for
a division algebra $D$ with a complete discrete valuation of rank
1, such that $Z(\ov D)$ is separable over $\ov F$,  $(1+M_D)\cap
D^{(1)} \subseteq D'$.  This crucial theorem was established with
a lengthy and rather complicated proof in \cite{platonov}. In
\cite{ershov}, Ershov states that the ``same'' proof will go
through for  tame valued division algebras over henselian fields.
However, this seems highly problematical, as Platonov's 
original proof used properties of maximal orders over 
discrete valuation rings which have no satisfactory analogues
for more general valuation rings.  
For the case of strongly tame division algebras, i.e.,  $\chr (\ov
F) \nmid [D:F]$, a short proof 
of the congruence theorem was  given in \cite{h} and another
(in the case of discrete  rank 1 valuations) in \cite{s}. In this
appendix, we provide a complete proof for the general
situation of a tame valued division algebra.

\begin{thm}[Congruence Theorem]\label{congru}
Let $F$ be a field with a henselian valuation $v$, and let~$D$ be a
tame $F$-cental division algebra. Then $(1 + M_D ) \cap D^{(1)}
\subseteq D'$.
\end{thm}

Tameness is meant here, as in the main body of the article, in the
weaker sense used in \cite{jw} and \cite{ershov}.
Among the several characterizations of tameness mentioned
in \S\ref{prel}, the ones we use here are that  $D$ is
tame if and only if $D$ is split by the maximal tamely ramified
extension of $F$, if and only if $\chr (\ov{F}) = 0$ or
$\chr(\ov{F}) = \ov p \neq 0$ and the $\ov{p}$- primary component
of $D$ is inertially split, i.e., split by the maximal unramified
extension of $F$.

The proof of the theorem will use the following well-known lemma:
\begin{lemma} \label{lemmaone} Let $D$ be a division ring with 
center $F$ and let $L$ be a
field extension of $F$ with $[L:F]= \ell $. If $a
\in D$ and $a \otimes 1 \in (D \otimes_F L)'$, then $a^{\ell} \in
D'$.
\end{lemma}
\begin{proof} The regular representation $L \rightarrow M_\ell(F)$
yields a ring monomorphism $D\otimes_F L\to M_\ell(D)$. 
Therefore, we have a composition of group homomorphisms
$$
(D\otimes_F L)^* \rightarrow \GL_{\ell}(D) \rightarrow D^*/D',
 \;\;\;
a\mapsto
\left(
\begin{smallmatrix}
a & 0 & \dots & 0\\
0 & a & \dots & 0\\
\vdots & \vdots & \ddots & \vdots\\
0 & 0 & \dots & a
\end{smallmatrix}
\right)_{\ell \times \ell} \mapsto a^{\ell}D',
$$
where the second map is the Dieudonn\'e determinant.
(See \cite{draxl}, \S20 for properties of the Dieudonn\'e
determinant.)
The lemma follows at once, since the image of the composition is 
abelian, so its kernel contains
$(D\otimes_F L)'$.
\end{proof}

Note that in the preceding lemma, there is no valuation
present, and  $D$ could be of infinite dimension over $F$.



\begin{proof}[Proof of Theorem ~\ref{congru}]  The proof is carried out in 
four steps.

{Step 1.} We prove the theorem if $D$ is inertially split of prime power
degree over $F$. This is a direct adaptation of Platonov's argument
in \cite{platonov} for discrete (rank~1) valuations. (When $v$ is discrete,
every tame division algebra is inertially split.)

Suppose $\ind(D) = p^k$, $p$ prime and $D$ is inertially split.
Then, $D$ has a maximal subfield~$K$ which is unramified over $F$
(cf.~\cite{jw}, Lemma~5.1, or \cite{wadval}, Th.~3.4) Take any 
${a \in (1+M_D) \cap D^{(1)}}$. We first 
push $a$ into $K$. Since $\ov K$ is separable over $\ov F$, there is
$y \in \ov K$ with $\ov K = \ov F( y)$. Choose
any $z \in V_K$ with $\ov z = y$. So $K= F(z)$, by dimension
count, as $\ov{F(z)} \supseteq \ov F(y)$. Note that $\ov{az}
= \ov z$ in $\ov D$. If $f$ is the minimal polynomial of $az$ over
$F$, then $f \in V_F[x]$ as $az \in V_D$, and $\ov z = \ov{az}$ is a
root of the image $\ov f$ of $f$ in $\ov F[x]$. We have $\deg (\ov
f) = \deg (f) = [F(az):F] \leq [K:F] = [\ov F( \ov z): \ov F]$.
Hence, $\ov f$ is the minimal polynomial of $\ov z$ over $\ov F$, so
$\ov z$ is a simple root of $\ov f$. By Hensel's lemma applied over
$K$, $K$ contains a root $b$ of~$f$ with $\ov b = \ov z$. Since $b$
and $az$ have the same minimal polynomial $f$ over $F$, by Skolem--Noether there is $t \in D^{\ast}$ with $b= t az t^{-1}$. So $az =
t^{-1} b t$. Then, 
$$
a  \ = \  t^{-1} b t z^{-1}  \ = \  (t^{-1} b t b^{-1})(b
z^{-1}).
$$ 
We have $b z^{-1} \in K$, as $b,z \in K$, and  $b z^{-1}
\equiv a \ (\text{mod }D')$; so, $\Nrd_D(b z^{-1}) = \Nrd_D(a) =1$, and $b
z^{-1} \in 1+M_D$, as $\ov b = \ov z$. Therefore, we may replace $a$
by $b z^{-1}$, so we may assume $a \in K$.

Let $N$ be the normal closure of $K$ over $F$, and let $G =
\Gal(N/F)$. 
Since $K$ is unramified over $F$ and the 
maximal unramified extension $F_{\text{nr}}$ of $F$ is Galois over 
$F$ (cf.~\cite{EP}, Th.~5.2.7, Th.~5.2.9, pp.~124--126), 
$N\subseteq F_{\text{nr}}$; so $N$ is also
unramified over $F$. 
Let $P$ be a $p$-Sylow subgroup of $G$ and let
$L = N^P$, the fixed field of $P$. Thus, $[L:F] = |G:P|$, which is
prime to $p$, and $N$ is Galois over $L$ with $\Gal(N/L) = P$.
Since $\gcd \big([L:F], \ind(D)\big) =1$, $D_1 = L \otimes_F D$
is a division ring and $K_1 = L \otimes_F K$ is a field 
with $K_1 \cong L\cdot K \subseteq N$. So, $K_1$ is unramified over $F$
and hence over $L$.
 We have $\Nrd_{D_1}(1
\otimes a) = \Nrd_D(a) = 1$ and $1 \otimes a \in 1 +M_D$, so if we
knew the result for $D_1$, we would have $1 \otimes a \in D_1'$. But
then by Lemma ~\ref{lemmaone}, $a^{[L:F]} \in D'$. But we also have
$a^{\ind(D)} \in D'$, since $\SK(D)$ is $\ind(D)$-torsion
(by \cite{draxl}, p.~157, Lemma~2 or Lemma~\ref{lemmaone} above
with $L$ a maximal subfield of $D$). Since
$\gcd\big([L:F], \ind(D)\big) =1$, it 
would follow that $a \in D'$, as desired.
Thus, it suffices to prove the result for $D_1$.

To simplify notation, replace $D_1$ by $D$, $K_1$ by $K$, $1
\otimes a$ by $a$, and $L$ by $F$. Because $F \subseteq K \subseteq
N$ with $N$ Galois over $F$ 
any subfield
$T$ of $K$ minimal over $F$ corresponds to a maximal subgroup of
$\Gal (N/F)$ containing $\Gal(N/K)$. Since 
$[N:F]$ is a power of $p$, by $p$-group theory such a
maximal subgroup is normal in $\Gal(N/F)$ and of index $p$. Thus,
$T$~is Galois over $F$ and $[T:F] = p$. So $\Gal(T/F)$ is a cyclic
group , say $\Gal(T/F) = \left< \sigma \right>$. Let $E= C_D(T)$,
so $F \sbeq T \sbeq K \sbeq E \sbeq D$. Note that $K$ is a maximal
subfield of $E$, since it is a maximal subfield of $D$.

Let $c = N_{K/T}(a)= \Nrd_E(a)$. Because $K$ is unramified over $T$
and $a \in V_K$, we
have $c \in V_T$ and $\ov c = N_{\ov K / \ov T}(\ov a) = N_{\ov K /
\ov T}(\ov 1) =\ov 1$, so $c \in 1+M_T$. We have $N_{T/F}(c) =
N_{T/F}(N_{K/T}(a)) = N_{K/F}(a) = \Nrd_D(a) = 1$. By Hilbert 90, $c
= b/\sigma(b)$ for some $b \in T$. This equation still holds if we
replace $b$ in it by any $F^{\ast}$-multiple of $b$. Thus, as $\Ga_T =
\Ga_F$ since $T$ is unramified over $F$, we may assume that $v(b) =
0$. But further, since $T$ is unramified and cyclic Galois over $F$,
its residue field $\ov T$ is cyclic Galois of degree $p$ over $\ov F$, with 
$\Gal(\ov T/ \ov F)= \left< \ov{\sigma} \right>$ where $\ov{\sigma}$ is
the automorphism of~$\ov T$ induced by $\sigma$ on $T$. In $\ov T$
we have $\ov b / \ov{\sigma}(\ov b) = \ov{b / \sigma(b)} = \ov c =
\ov 1$. Therefore, $\ov b$ lies in the fixed field of $\ov{\si}$ in $\ov
T$, which is $\ov F$ . Hence, there is $\eta \in V_F$ with
$\ov{\eta} = \ov b$ in $\ov T$. By replacing $b$ by $b \eta^{-1}$,
we can assume $ \ov b = \ov 1$, i.e., $b \in 1 +M_T$.

Since $K$ is unramified and hence tame over $T$,  Prop. ~\ref{normsurj} 
shows $N_{K/T}(1+M_K) = 1+M_T$. So, there is $s \in 1+M_K$ with
$N_{K/F}(s) = b$. Now, by Skolem--Noether, there is an inner
automorphism $\varphi$ of $D$ such that $\ph(T) =T$ and $\ph |_T =
\si$. Since $E= C_D(T)$, we have $\ph$ is a (non-inner) automorphism
of $E$, and $\ph(K)$ is a maximal subfield of $E$ (since $K$ is a
maximal subfield of $E$). We have $\Nrd_E(\ph (s)) = N_{\ph (K)/\ph
(T)} (\ph(s)) = \ph \big(N_{K/T}(s)\big) = \si (b)$. Thus, 
$$
\Nrd_E (s/ \ph(s))  \ = \  b / \si (b)  \ = \  c.
$$ 
Now, as $\ph$ is inner, there is $u \in
D^{*}$ with $\ph (s) = u s u^{-1}$. So, $\ph (s) \in 1 + M_D$. Let
${a' = a/ (s/ \ph(s)) = a\big/ (sus^{-1} u^{-1}) \in E}$. So 
$a' \equiv a \ (\text{mod }D')$. 
But further, $a' \in E \cap (1+M_D) = 1+M_E$ (as $a,s,
\ph(s) \in (1+M_D)\cap E\,)$. Also, 
$$ 
\Nrd_E (a')  \ = \  \Nrd_E (a) \big / \,\Nrd_E (s/ \ph (s)) 
 \ = \  N_{K/T}(a) / c  \ = \  1.
$$
Since $[E:T] < [D:F]$ and $E$ is inertially split over $T$ (since it
is split by its maximal subfield~$K$ which is unramified over $T$),
by induction on index the theorem holds for $T$ over~$E$. Hence,
$a' \in E'$. Since $a \equiv a' \ (\text{mod }D')$, we thus have $a \in D'$,
as desired. This completes the proof of Step 1.

\smallskip

{Step 2.} The theorem is true if $D$ is strongly tame over $F$, i.e.,
$\chr (\ov F) \nmid [D:F]$.
This has a short proof given in \cite{h} and another
 (in the case of discrete valuation of rank 1) in \cite{s},
Lemma 1.6.  For the convenience of the reader, we recall the 
argument from \cite{h}:

Let $n = \ind(D)$, so $\chr(\ov F)\nmid  n$. Take any $s\in D^*$, and let 
$f = x^k +c_{k-1}x^{k-1} + \ldots + c_0\in F[x]$ be the minimal
polynomial of $s$ over $F$.  By applying 
 the Wedderburn factorization theorem to~$f$
(see \cite{lam}, (16.9), pp.~251--252, or Appendix A above), we see 
that there  exist $d_1, \ldots, d_k
\in D^*$ with $(-1)^k c_0 = (d_1sd_1^{-1}) \ldots (d_ksd_k^{-1})$.  
Hence, as $D^*/D'$ is abelian,   
\begin{equation} \label{npower}
\Nrd_D(s) \ = \  [(-1)^kc_0]^{n/k} \ \equiv  
\ \big[s^k(d_1sd_1^{-1}s^{-1}) \ldots (d_ksd_k^{-1}s^{-1})\big]^{n/k}
 \ \equiv  \ s^n \ \ (\text {mod } D').
\end{equation}
Now, take any $a \in 1+M_D$ with $\Nrd_D(a) = 1$.  Since 
$\chr(\ov F) \nmid n$, Hensel's Lemma applied over $F(a)$ shows 
that there is $s \in 1+M_F(a) \subseteq 1+M_D$ with $s^n = a$.
Then, $\Nrd_D(s) = 1+m \in 1+M_F$ by Cor.~\ref{Dnormsurj}.  But, 
$$
(1+m)^n \ = \ \Nrd_D(a^n) \ = \ \Nrd_D(s) \ = \ 1.
$$
If $m \ne 0$, then we have $1 = (1+m)^n = 1 +n\, m + r$
with $v(r) \ge 2v(m)$, which would imply that $v(n\, m) = 
v(r) > v(m)$.  This cannot occur since $\chr(\ov F) \nmid  n$;
hence, $m = 0$.
Thus, by~\eqref{npower}
$$ 
a \ = \ s^n \ \equiv \ \Nrd_D(s) \ = \ 1+m  \ = \ 1 \ (\text{mod }
D'),
$$
i.e., $a\in D'$.  This completes Step 2.

\smallskip 

{Step 3.} Suppose $D= P \otimes_F Q$, where $\gcd(\ind (P),
\ind(Q)) =1$, and suppose the theorem is true for $P_L$ and $Q_L$ 
for any subfield $L$ of $D$, $L \speq F$, where 
$P_L$ (resp.~$Q_L$) is the division algebra Brauer equivalent
to $L\otimes_F P$ (resp.~ $L\otimes _F Q$). Then we 
show using Prop. ~\ref{prophelp} below that the theorem is true 
for $D$.

Let $L$ be a maximal subfield of $P$, and let  
$C = C_D(L)$. Then,   $C\cong C_L(P) \otimes_F Q= L \otimes_F Q$;
since $C$ is a division ring, $Q_L\cong C$.
Also, 
$$
L\otimes_F D  \ \cong  \ (L\otimes _F P) \otimes _L (L\otimes _F Q)
 \ \cong  \ M_\ell(L) \otimes_L C  \ \cong \  M_\ell(C),
$$ 
 where $\ell = [L:F] = \ind(P)$. 
Take any $a \in (1+M_D)\cap D^{(1)}$. For $1
\otimes a \in L \otimes D = M_{\ell}(C)$,  Prop.~\ref{prophelp}
shows that
there is $c \in 1+M_C$ with $\ddet(a) \equiv c \ (\text{mod }C')$,
where $\ddet$ denotes the Dieudonn\'e determinant. Then, 
$$
1 \ =  \ \Nrd_D(a)  \ = \  \Nrd_{M_{\ell}(C)}(1 \otimes a) 
 \ =  \ \Nrd_{C} (\ddet (1\otimes a))  \ =  \ \Nrd_C(c).
$$ 
Hence, $c \in (1+M_C) \cap C^{(1)}$ which
 lies in $C'$ by hypothesis as $C = Q_L$. That is, ${\ddet (1\otimes
a) = 1 \in C^*/C'}$. Hence, $1 \otimes a \in \ker (\ddet) =
(L\otimes_F D)'$. Therefore, by Lemma ~\ref{lemmaone}, $a^{\ell} \in
D'$. Likewise, we can take a maximal subfield $K$ of $Q$, and by
looking at $1 \otimes a \in K \otimes_F D$, we obtain $a^k \in D'$
where $k = [K:F] = \ind(Q)$. 
Since $\gcd(\ell, k) = 1$, it follows that $a \in
D'$, completing Step 3.

\smallskip

Step 4.  We now prove the theorem in full.  Let $F$ be  a henselian 
field, and 
 let $D$ be a tame $F$-central division algebra. 
If $\chr (\ov F) =0$, then $D$ is strongly tame over
$F$, so the theorem holds for $D$ by Step 2. If $\chr (\ov F) = \ov
p\ne 0$ we have $D \cong  P \otimes_F Q$ where $P$ is the 
$\ov p$-primary
component of $D$ and $Q$ is the tensor product of all the other
primary components of $D$.  So, $\gcd(\ind(P), \ind(Q)) = 1$.
For any subfield $L$ of $F$, $Q_L$ is
tame over $L$ with $\ind(Q_L) | \ind (Q)$, which is prime to $\ov
p$. So, $Q_L$ is strongly tame over $L$, and the theorem holds for
$Q_L$ by Step~2. On the other hand, $P_L$ is tame over $L$ and
$\ind(P_L)$  is a power of $\ov p$; hence, $P_L$~is
inertially split. Hence, by Step 1 the theorem holds for $P_L$.
Thus, by Step 3 the theorem holds for $D$.
\end{proof}

The following proposition will complete the proof of the Congruence
Theorem.

\begin{prop}\label{prophelp}
Let $F$ be a henselian valued field, and 
let $D$ be an $F$-central division
algebra which is defectless over $F$. Let $L$ be a field, $F \sbeq L
\sbeq D$, and let $C= C_D(L)$, so $L \otimes_F D \cong M_{\ell}(C)$
where $\ell = [L:F]$. Take any $a \in 1+M_D$. Then for $1 \otimes a
\in L \otimes_F D \cong M_{\ell}(C)$, 
$$
\ddet (1 \otimes a) \in 1+M_C \ (\mathrm{mod} \ C'),
$$ 
where $\ddet$ denotes the
Dieudonn\'{e} determinant.
\end{prop}

\begin{proof}
$D$ is an $L$-$D$ bimodule via multiplication in $D$. Hence (as $L$
is commutative) $D$ is  a right $L \otimes_F D$-module, 
with module
action  given by $a (\sum \ell_i \otimes d_i) = \sum \ell_i a
d_i$. $D$ is a simple right $L \otimes_F D$-module, since it is
already a simple right $D$-module. Hence, by Wedderburn's Theorem,
$L \otimes_F D \cong \End_{\De}(D)$, where $\De = \End_{L \otimes_F
D}(D)$ (acting on $D$ on the left). Since (for $D$ acting on $D$ on
the right) $\End_D(D) \cong D$ (elements of $D$ acting on $D$ by
left multiplication) $\End_{L \otimes_F D}(D)$ consists of left
multiplication by elements of $D$ which commute with the left action
of $L$ on $D$, i.e., $\De \cong C_D(L) = C$. So, $L \otimes_F D \cong
\End_{\De}(D) \cong \End_C (D) \cong M_{\ell}(C)$ where $\ell =
[D:C]=[L:F]$. The last isomorphism is obtained by choosing a base
$\{ b_1, \ldots , b_{\ell} \}$ of $D$ as a left $C$-vector space $(D
= Cb_1 \oplus \ldots \oplus Cb_{\ell})$ and writing the matrix for
an element of $L \otimes_F D$ acting $C$-linearly on $D$ (on the
right) relative to this base, with matrix entries in $C$.

Because $D$ is defectless over  $F$, $D$ is also defectless over
$C$, i.e., $[D:C] = [\gr(D):\gr(C)]$; thus, the valuation $w$ on $D$ extending
$v$ on $F$ is a $w|_C$-norm by \cite{RTW}, Cor.~2.3. 
This means that we can choose our base
$\{b_1, \ldots, b_{\ell}\}$ to be a splitting base for $w$ over
$w|_C$, i.e., satisfying, for all $c_1, \ldots , c_{\ell} \in C$,
\begin{equation}\label{mindef}
\textstyle w\big( \sum\limits_{i=1}^{\ell} c_i b_i\big) \ = \ 
\min\limits_{1\leq i \leq \ell} \big(w(c_i) +
w(b_i)\big).
\end{equation}
Let $\gamma_i = w(b_i)$ for $1\le i\le \ell$.

Let
\begin{align*}
R    \ &=  \ \{A = (a_{ij})\in M_\ell(C) : w(a_{ij}) \ge 
\gamma_i - \gamma_j\text{ for all  }i,j\}; 
\qquad\qquad\qquad\qquad\qquad\qquad\\
J   \ & =   \ \{A = (a_{ij})\in M_\ell(C) : w(a_{ij}) > 
\gamma_i - \gamma_j\text{ for all  }i,j\};\\
1+J   \ &=  \ \{I_\ell + A : A \in J \}, \ \ 
\text{where $I_\ell\in M_\ell(c)$ is the identity matrix}.
\end{align*}
Because $w$ is a valuation, it is easy to check that $R$ is a subring 
of $M_\ell(C)$ and $J$ is an ideal of~$R$.  Therefore, $1+J$ is closed 
under multiplication.  Take any $f\in \End_C(D)$ (which acts on~$D$ 
on the right), and let $A = (a_{ij})$ be the  matrix of $f$ relative 
to the $C$-base $\{b_1, \ldots b_\ell\}$ of~$D$, i.e., 
$b_if = \sum _{j=1}^\ell a_{ij}b_j$ for all $i$.  So, 
$$
w(b_if) \ = \ w\big(\textstyle \sum\limits_{j=1}^\ell a_{ij}b_j\big) 
 \ =  \ \min\limits_{1\le j\le \ell}\big(w(a_{ij}) + \gamma_j\big).
$$
Thus, $w(b_if) \ge w(b_i) = \gamma_i$ iff 
$w(a_{ij}) \ge\gamma_i - \gamma_j$ for 
$1\le j\le \ell$.  
From this it is clear that $A= (a_{ij}) \in R$ iff 
$w(b_if) \ge w(b_i)$ for all $i$. Analogously, $A\in J$ iff 
$w(b_if) > w(b_i)$
for all $i$. 

Now, take any $u \in 1+M_D$, say $u = 1+m$ with $m\in M_D$.  
Then, $1\otimes m \in L\otimes_F D$ corresponds to $\rho_m\in 
\End_C(D)$, where $d\rho_m = dm$ for all $d\in D$.  Let 
$S\in M_\ell(C)$ be the matrix for $\rho_m$.  Since $w(m) >0$,
we have 
$$
w(b_i\rho_m) \ = \ w(b_im) \ = \ w(b_i) + w(m)  \ > \ w(b_i)
\quad\text{for all $i$}.
$$
Hence, $S\in J$, as we saw above.  

 {\it Claim.} For any matrix $T \in 1+J$, we have
$\ddet(T) \in 1+M_C\ (\text{mod} \ C')$. 

The Proposition follows at once from this claim, since the 
matrix for $1\otimes (1+m)$ is ${I_\ell + S \in 1+J}$. 

{\it Proof of Claim.}  Take $T\in 1+J$.
The idea is that the process of bringing $T$ to 
upper triangular form by row operations is 
carried out entirely within $1+J$.  Write 
$T = I_\ell + Z$ with $Z= (z_{ij}) \in J$.  
So, $w(z_{ii}) >\gamma_i-\gamma_i = 0$ for all $i$, 
i.e., $z_{ii} \in M_C$.  Thus, for all $i,j$, we have
$$
t_{ii} \ = \  1+z_{ii}  \,\in \, 1+M_C \qquad\text{and}\qquad 
t_{ij} \ = \ z_{ij}, \text{  \, so  \, $w(t_{ij}) \ > \ 
\gamma_i - \gamma_j$  \, when  \, $i\ne j$}. 
$$

Fix $k$ with $1\le k \le \ell -1$.  Since $t_{kk} \in 1+M_C$,
$w(t_{kk}) = 0$, so $t_{kk}\ne 0$.  Let $Y = (y_{ij}) \in M_\ell(C)$
be the matrix for the row operations to bring $0$'s to all entries
in the $k$-th column of $T$ below the main diagonal, i.e., 
the $i$-th row of $YT$ is: $\text{(the $i$-th row of $T$) $-$ 
($t_{ik}t_{kk}^{-1}\cdot$ the $k$-th row of $T$)}$
for $k<i\le \ell$ (with the first $k$ rows unchanged). 
So, $y_{ii} = 1$ for all $i$; $y_{ik} = -t_{ik}t_{kk}^{-1}$
for our fixed $k$ and all $i$ with $k<i\le \ell$; and 
$y_{ij} = 0$ otherwise. For $i>k$,we have 
$$
w(y_{ik}) \ = \   w(t_{ik}) - w(t_{kk}) \  >
 \ \gamma_i - \gamma_k.
$$
Hence, $Y \in 1+J$ and $Y$ is a unipotent lower triangular matrix.
Since $1+J$ is closed under multiplication, we have $YT\in 1+J$.
To bring $T$ to upper triangular form we apply the row operations 
successively 
for columns $1$ to $\ell -1$. We end up with an upper triangular
matrix $T' = Y_{\ell-1} Y_{\ell-2}\ldots Y_2Y_1 T \in 1+J$,
where each $Y_k\in 1+J$ is the matrix for zeroing the $k$-th column 
as described above, but applied to the matrix $Y_{k-1}\ldots Y_1T 
\in 1+J$ (not to $T$).  Say $T'= (t_{ij}')$. Each $Y_k$ is unipotent and 
lower triangular, so $\ddet(Y_k) = 1\in C^*/C'$, 
So, $\ddet(T') = \ddet(Y_{k-1})\ldots \ddet(Y_1) \ddet(T) = 
\ddet(T)$ in $C^*/C'$.  Since $T'$ is upper triangular with 
each $t_{ii}' \in 1+M_C$, we have 
$$ \ddet (T)   \ = \  \ddet (T')  \ = \  t'_{11}
\ldots t'_{ \ell \ell}  \, \in \,  1+M_C \textrm{ (equality modulo $C'$),}
$$
proving the Claim.
\end{proof}

\end{document}